\documentclass[11pt,a4paper]{article}

\usepackage[T1]{fontenc}
\usepackage[utf8]{inputenc}
\usepackage{csquotes}

\usepackage{amsfonts, amsmath, amssymb, amsthm, dsfont, mathrsfs, latexsym, euscript}

\usepackage[
backend=biber,
style=ieee,
sorting=nyt,
url = true
]{biblatex}
\usepackage{tikz}
\usepackage{xcolor}
\usepackage{caption, subcaption}
\usepackage{hyperref}
\usepackage[left=2cm,right=2cm,top=2cm,bottom=2cm]{geometry}

\def\DD{\mathbb{D}}
\def\EE{\mathbb{E}}
\def\FF{\mathbb{F}}

\def\LL{\mathbb{L}}

\def\NN{\mathbb{N}}

\def\PP{\mathbb{P}}

\def\RR{\mathbb{R}}

\def\Bcal{\mathcal{B}}
\def\Ccal{\mathcal{C}}

\def\Fcal{\mathcal{F}}

\def\Ical{\mathcal{I}}

\def\Ncal{\mathcal{N}}

\def\Scal{\mathcal{S}}
\def\Tcal{\mathcal{T}}

\def\Wcal{\mathcal{W}}

\def\Zcal{\mathcal{Z}}

\newcommand{\bdF}{\mathbf{\Delta F}}

\newcommand{\bG}{\mathbf{G}}

\newcommand{\bu}{\mathbf{u}}
\newcommand{\bx}{\mathbf{x}}

\newcommand{\bsig}{\mathbf{\Sigma}}

\newcommand{\lip}{\mathrsfs{L}}

\newcommand{\eps}{\varepsilon}
\def\E#1{\mathbb{E}\left[ #1 \right]}

\def\ind#1{\mathds{1}_{\{#1\}}}
\def\1{\mathds{1}}

\newtheorem{theorem}{Theorem}[section]
\makeatletter
\renewcommand{\thetheorem}{%
  \ifnum\c@subsection=0
    \thesection.\number\c@theorem
  \else
    \thesubsection.\number\c@theorem
  \fi}
\makeatother

\newtheorem{corollary}[theorem]{Corollary}
\newtheorem{lemma}[theorem]{Lemma}
\newtheorem{proposition}[theorem]{Proposition}
\newtheorem{definition}[theorem]{Definition}
\newtheorem{assumption}[theorem]{Assumption}
\newtheorem{notation}[theorem]{Notation}
\newtheorem{remark}[theorem]{Remark}

\numberwithin{equation}{section}
\makeatletter
\renewcommand{\theequation}{%
  \ifnum\c@subsection=0
    \thesection.\number\c@equation
  \else
    \thesubsection.\number\c@equation
  \fi}
\makeatother

\def\z{\mathrsfs{Z}}
\DeclareSymbolFontAlphabet{\mathrsfs}{rsfs}
\DeclareMathOperator*{\esssup}{ess\,sup}

\title{Limit theorems for non linear (compound marked) Hawkes processes}
\author{Benjamin Massat\footnote{UPS, IMT UMR CNRS 5219, Universit\'e de Toulouse, 118 route de Narbonne 31062 Toulouse Cedex 9 France. \; Email: \texttt{benjamin.massat@math.univ-toulouse.fr}}  }

\begin{document}
\maketitle

\begin{abstract}
    In this article, we fill a gap in the literature on Hawkes processes. In particular, we derive a CLT for a non linear compound marked Hawkes process. We also provide an upper bound on the convergence rate using the functional $1$-Wasserstein distance. This result is obtained by discretizing the time line and reducing the problem to the quantification of the distance between finite marginal vectors, as well as between the discretized process and the original one.
\end{abstract}


\section{Introduction}

Introduced by Hawkes in \cite{hawkes_spectra_1971}, the Hawkes process is a self-exciting point process whose intensity strongly depends on the past events. Whereas it was defined in a linear context, this paper will focus on non linear compound marked Hawkes process, denoted $L$. This process is defined through a sum whose last term is determined by a marked linear Hawkes process. Mathematically, we can write:
\begin{equation}
\label{eq: model_marked_compound}
        L_t := \sum^{H_t}_{i=1} g(X_i), \quad \text{with} \quad \left\{ \begin{array}{ccl}
             H_t &=& \int_{(0,t]\times \RR_+} \1_{\theta \leq \lambda_s }N(ds,d\theta) \\
            \lambda_t &=& h \left(\mu + \int_{(0,t)} b\left(X_{H_s}\right) \phi(t-s) dH_s \right)
        \end{array} \right. , \quad t\in \RR_+ .
    \end{equation}
where $h: \RR\to \RR_+$, $\mu \in \RR$, $\phi: \RR_+ \to \RR$, $b:\RR \to \RR$, $g: \RR\to \RR_+$ and $\left(X_i\right)_{i\in\NN}$ is a sequence of independent random variables with common law $\vartheta$.\\
Note that by taking $g=b\equiv 1$, we then get the classical non linear Hawkes process. Our model can be found in \cite{khabou_sur_2022} by Khabou for a particular kernel: exponential $\phi(x) = \alpha e^{-\beta x}$ and Erlang $\phi(x) = \alpha x e^{-\beta x}$. Another interesting case is the special one that we obtain by taking $h=Id$, $\phi$ positive and $\mu>0$. Indeed, this model is referred as the linear Hawkes process. Whereas the linear Hawkes process was the original one, non linearity increases modeling flexibility by allowing for saturation, inhibition, and offers more complex excitation dynamics, which cannot be captured by the linear Hawkes. For these reasons, non linear (marked) Hawkes processes have gained popularity: Brémaud and Massoulié in \cite{bremaud_stability_1996} give stability results;  Brémaud, Nappo, Torrisi in \cite{bremaud_rate_2002} find convergence rate to stationarity; Zhu in \cite{zhu_nonlinear_2013} gives the first functional central limit theorem (fCLT) for instance for linear (marked) Hawkes process and for non linear (unmarked) Hawkes process.

\subsection{Limits theorems for Hawkes processes}

In 2013, Bacry, Delattre, Hoffmann, and Muzy \cite{bacry_limit_2013} established a law of large numbers and a functional central limit theorem (fCLT) for linear multivariate Hawkes processes observed over a time interval $[0,T]$ as $T \to +\infty$:
\begin{equation}
\label{eq: fCLT}
    \left(\frac{H_{tT}- \int_0^{tT} \lambda_s ds}{\sqrt{T}}\right)_{t\in[0,1]} \xrightarrow[]{T\to +\infty} \left( \sigma B_t\right)_{t\in [0,1]},
\end{equation}
where $B_t$ is a Brownian motion. This result was later extended by Zhu \cite{zhu_nonlinear_2013} to nonlinear Hawkes processes. We can also cite Cattiaux, Colombani and Costa whose work (see \cite{cattiaux_limit_2022}) concerns nonlinear Hawkes processes with signed reproduction function (or memory kernel) thus exhibiting both self-excitation and inhibition. In the same setting, Costa, Maillard and Muraro give characterization of stability of a discrete-time Hawkes process with inhibition and memory of length two (see \cite{costa_almost_2024} for more details on the matter).

The first bound on the convergence rate of such processes was established by Torrisi \cite{torrisi_gaussian_2016} for nonlinear Hawkes processes. However, this bound does not tend to zero as $T \to +\infty$.  A major issue is the quantification of functional Central Limit Theorems; this is adresses in \cite{besancon_diffusive_2024} by Besancon, Coutin and Decreusefond using other techniques. Besides, Besançon, Coutin, Decreusefond, and Moyal, in \cite{besancon_diffusive_2024}, quantified convergences to diffuse limits of Markov processes and long-time behavior of Hawkes processes. In this paper, they study the convergence of functionals of a one-dimensional compensated Poisson measure towards a functional of a Brownian motion. For the Hawkes process, they establish the convergence of the linearly interpolated renormalized Hawkes process towards the linearly interpolated Brownian motion.
Between 2018 and 2022, a series of works by Hillairet, Huang, Khabou, Privault, and R\'eveillac (see \cite{hillairet_malliavin-stein_2022}, \cite{khabou_malliavin-stein_2021}, \cite{privault_stein_2018}) analyzed the convergence rate of the linear Hawkes process in 1-Wasserstein distance of one-dimensional marginals, obtaining the bound  
\begin{equation}
\label{eq: maj_intro}
    \Wcal_1 \left( \frac{H_T - \int_0^T \lambda_t dt}{\sqrt{T}}, \Ncal\left(0,\sigma^2 \right)\right) \leq \frac{C}{\sqrt{T}}.
\end{equation}
Their work relies on the Poisson imbedding representation coupled with a taylor-made Malliavin calculus for Hawkes processes completing the approach of \cite{torrisi_gaussian_2016} for this class of processes. Recall that Poisson imbedding consists in representing a counting process as a solution of an SDE driven by a Poisson random measure. We refer to Brémaud and Massoulié \cite{bremaud_stability_1996} for a discussion on Poisson imbedding and to Peccati, Solé, Taqqu, and Utzet \cite{peccati_steins_2010} for Malliavin calculus on the Poisson space.

This methodology is also used in \cite{coutin_normal_2024}, where Coutin, Massat, and Réveillac extend the previous results by obtaining new upper bounds in Wasserstein distance between a functional of point processes and a Gaussian distribution. Their results apply to nonlinear Hawkes processes and discrete linear Hawkes processes (see \cite{kirchner_hawkes_2016} and \cite{quayle_etude_2022} for details on these processes). They also improve the convergence theorem for nonlinear Hawkes processes from \cite{zhu_nonlinear_2013} by relaxing some assumptions.

Furthermore, Horst and Xu \cite{horst_functional_2024} established functional and scaling limit theorems for Hawkes processes under minimal conditions on their kernels using Fourier analysis. In particular, they characterized the limiting behavior of subcritical ($ \left\| \phi \right\|_1 < 1$) and critical ($\left\| \phi \right\|_1 = 1$) Hawkes processes. More recently, Xu \cite{xu_scaling_2024} also studied the scaling limit of multivariate Hawkes processes.\\

As mentioned in the begin of the article, Hawkes processes have gain in attractiveness over the last two decades. With this come variations of Hawkes process. We have already developed the case of linear Hawkes process (see \cite{bremaud_hawkes_2001} or \cite{zhu_nonlinear_2013} for instance) or Hawkes process in critical regime (see \cite{jaisson_limit_2015}, \cite{horst_functional_2024} or \cite{liu_scaling_2024} for instance). Let us focus on an other type of extension of the definition of Hawkes processes: CLTs on marked Hawkes process. Longtime behavior of these Hawkes processes have been studied by Karabash and Zhu in \cite{karabash_limit_2015} . They provide CLT for linear marked Hawkes process and prove large deviation result for their model. Later, \cite{khabou_sur_2022}, for the sake of generalization, obtained similar results without supposing that the mark are completely independent from the claim (whereas it is supposed in \cite{karabash_limit_2015}). However, the work of \cite{khabou_sur_2022} focus only on linear model with an exponential or Erlang kernel. Another work written by Khabou and Torrisi \cite{khabou_gaussian_2025} provides new explicit Wasserstein and Kolmogorov bounds for normal approximation of random variables in the first Poisson chaos. They thereby apply their result to Hawkes-type point processes with a non-Poissonian offspring distribution and so improved the different achievement that have been made by \cite{hillairet_malliavin-stein_2022} \cite{khabou_malliavin-stein_2021} and \cite{khabou_normal_2024} for instance.

\subsection{Our contibution}

The goal of this paper is to provide a functional central limit theorem for non linear marked Hawkes process in the space of right continuous with left limit functions, denoted $\DD([0,1],\RR)$, endowed with the uniform norm. Indeed, to our knowledge, their is no existing result that deal with non linearity, composition and marking.\\ 
Moreover, we will provide quantification (in a functional sens) of Equation \eqref{eq: fCLT} which will be the first convergence rate for non linear Hawkes processes by taking $b\equiv 1$. The metric used is thereupon the 1-Wasserstein distance $\Wcal_1$ defined as follows:
\begin{align*}
    \Wcal_1(X,Y) := \sup_{f\in \left(\lip_1^{\DD}, \left\|\cdot \right\|_{\LL^\infty(0,1)} \right)} \E{f(X)}-\E{f(Y)},
\end{align*}
where $\lip_1^{\DD} = \left\{ f: \DD([0,1],\RR) \to \RR \mid \left|f(x)-f(y)\right|\leq \left\|x-y \right\|_{\LL^\infty(0,1)} := \sup_{t\in [0,1]} \left| x_t-y_t\right| \right\}$. Moreover, we prove the functional convergence \eqref{eq: fCLT} without monotony assumptions on the kernel $\phi$ or the non-linear function $h$. (see Theorem \ref{thm: main}). This was possible by building upon the approaches developed in both \cite{coutin_donskers_2020} by Coutin and Decreusefond, and \cite{nourdin_multivariate_2010} by Nourdin, Peccati, and Réveillac and so we extend the marginal result presented in \cite{coutin_normal_2024}. Specifically, we combine Stein’s method with Malliavin calculus, effectively bridging these two frameworks. This approach, commonly known as the Nourdin–Peccati methodology, was first introduced in \cite{nourdin_steins_2009} and has since established itself as a fundamental tool for assessing convergence rates in limit theorems.

We proceed as follows. The Poisson imbedding and element of Malliavin's calculus are presented in Section \ref{sec: notations}. The main result, which is the first functional central limit theorem with the first functional bound regarding non linear compound marked Hawkes process, is collected in Section \ref{sec: main result}. Finally, technical lemmata are postponed in Section \ref{sec: tech_lem}.

\section{Notation and preliminaries}
\label{sec: notations}
\subsection{General notations}
We denote by $\NN$ (resp. $\NN^*$) the set of non-negative (resp. positive) integers, that is, $\NN := \{0, 1, 2, \dots\}$ (resp. $\NN^* := \{1, 2, \dots\}$). Similarly, we define the sets of non-negative and positive real numbers as $\RR_+ := [0, +\infty)$ and $\RR_+^\ast := (0, +\infty)$, respectively. 

Regarding function spaces, for any $A \in \Bcal(\RR)$ and $p \in \NN^*$, we define the Lebesgue space  
$$ \LL^p(A) = \left\{ f:A \to \RR \mid \|f\|_{\LL^p(A)} := \left(\int_A |f(t)|^p dt \right)^{1/p} < +\infty \right\}.$$  
For $p = \infty$, we set the set of a.e. bounded functions
$$ \LL^{\infty}(A) = \left\{ f:A \to \RR \mid \|f\|_{\LL^{\infty}(A)} := \esssup_{t\in A} |f(t)| < +\infty \right\}. $$  
In the specific case $A = \RR_+$, we use the shorthand notation $\LL^p := \LL^p(\RR_+)$ for any $p \in \NN^* \cup \{\infty\}$.
The function space of càdlàg function from $[0,1]$ to $\RR$ is denoted by $\DD$.
We also introduce for $p\in \NN^\ast$ the sequence spaces $\ell^p$ and $\ell^\infty$ defined by
\begin{align*}
    \ell^p &:= \left\{ \bu = (u_n)_{n\in \NN} \subset \RR^{\NN} \mid \left\| \bu\right\|_p  := \left( \sum_{n=0}^{\infty} |u_n|^p \right)^{1/p}< +\infty\right\},\\
    \ell^\infty &:= \left\{ \bu = (u_n)_{n\in \NN} \subset \RR^{\NN} \mid \left\| \bu\right\|_\infty  := \sup_{n\in \NN} \left| u_n\right|<+\infty\right\}
\end{align*}
Moreover, we make use of different sets of Lipchitz functions with different Lipchitz constants and different norms. These sets will be generalized as follow
\begin{align*}
    \left( \lip_L^E, \left\| \cdot \right\|\right) := \left\{ f: E\to \RR \mid \forall (x,y)\in E^2, \, \left| f(x)-f(y)\right| \leq L \left\| x-y\right\|  \right\}
\end{align*}

We also make use of the Hilbert–Schmidt inner product on the class of $n\times n$ matrice. We denote it by $\left\langle \cdot, \cdot\right\rangle_{H.S.}$ and it is defined between two matrices $A$ and $B$ by $\left\langle A,B\right\rangle_{H.S.} := \text{Tr}\left( A B^{t}\right)$

\subsection{Elements of stochastic analysis on the Poisson space}

In this section, we begin by recalling the basic ingredients of Malliavin's calculus that will serve as the foundation of our analysis. Our presentation follows, in spirit, the framework of \cite{khabou_sur_2022}. In particular, the Poisson random measure will be described using three variables: $t$ for the jump times, $x$ for the jump magnitudes with distribution $\vartheta$, and an auxiliary parameter $\theta$, which plays a role in the imbedding representation used later.\\
We start by defining the space of configurations
$$ \Omega:=\left\{\omega=\sum_{i=1}^{n} \delta_{(t_{i},\theta_i, x_i)} \mid t_0 < t_1 < \cdots < t_n, \; (\theta_i, x_i)_{i=1,\dots,n} \in \left(\RR_+ \times \RR \right)^n, \; n\in \NN \cup\{+\infty\} \right\}.$$
Each realization of a counting process is represented as an element $\omega$ in $\Omega$ which is a $\NN$-valued measure on $\RR \times \RR_+ \times \RR$. Let $\Fcal^N$ be the $\sigma$-field associated to the vague topology on $\Omega$, and $\PP$ be a probability measure under which the random measure $N$ defined as:
$$ N(B)(\omega):=\omega(B), \quad B \in \Bcal \left( \RR\times \RR_+ \times \RR\right),$$
is a random measure with intensity $1$ (so that $N(B)$ is a Poisson random variable with intensity $\pi\otimes \vartheta (B)$ for any $B\in \Bcal \left( \RR \times \RR_+ \times \RR\right)$ and where $\pi$ denotes the Lebesgue measure on $\RR^2$). We set $\FF^N:=(\Fcal_t^N)_{t\in \RR}$ the natural history of $N$, that is $$\Fcal_t^N:=\sigma \left(N( \Tcal  \times B), \; \Tcal \subset \Bcal((-\infty,t]), \; B \in \Bcal(\RR_+ \times \RR) \right).$$ Let also, $\Fcal_\infty^N:=\lim_{t\to+\infty} \Fcal_t^N$. The expectation with respect to $\PP$ is denoted by $\E{\cdot}$. For $t\in \RR$, we denote by $\EE_t[\cdot]$ the conditional expectation $\E{\cdot \vert \Fcal_t^N}$.\\\\
\noindent
We recall some elements of stochastic analysis on the Poisson space, especially the shift operator, the Malliavin derivative and its dual operator: the divergence. 
The following essentially come from Hillairet et al. in \cite{hillairet_malliavin-stein_2022}. Their paper provides some element of the  Malliavin calculus for the marked Hawkes process (which corresponds to the case $g\equiv 1$. Besides, note that their result do not change for our setting and that is why we decide to not provide the proofs here.
\begin{definition}[Shift operator]
\label{definition:shift}
    We define for $(t,\theta,x)$ in $\RR \times \RR_+\times \RR$ the measurable map $\eps_{(t,\theta,x)}^+ : \Omega \to \Omega $ where for any $A$ in $\Bcal(\RR \times \RR_+\times \RR)$
    $$(\eps_{(t,\theta,x)}^+(\omega))(A) := \left\lbrace \begin{array}{l} \omega(A \setminus \{(t,\theta,x)\}) +1, \quad \textrm{if } (t,\theta,x)\in A,\\\omega(A \setminus \{(t,\theta,x)\}) , \quad \textrm{else.}\end{array}\right. $$
\end{definition}

\begin{lemma}
    \label{lemma:mesur}
    Let $t \in \RR$ and $F$ be an $\Fcal_t^N$-measurable random variable. Let $v > t$, $\theta\geq 0$ and $x\in \RR$. It holds that 
    $$ F\circ\eps_{(v,\theta,x)}^+ = F, \quad \PP-a.s. $$
\end{lemma}

\begin{definition}[Malliavin derivative]
    For $F$ in $\LL^2(\Omega,\Fcal_\infty^N,\PP)$, we define $D F$ the Malliavin derivative of $F$ as 
    $$ D_{(t,\theta,x)} F := F\circ \eps_{(t,\theta,x)}^+ - F, \quad (t,\theta,x) \in \RR\times \RR_+ \times \RR.  $$
\end{definition}

\noindent The following definition is a by-product of \cite[Theorem 1]{picard_formules_1996} (see also \cite{nualart_anticipative_1990}).

\begin{definition}
    Let $\Ical$  be the sub-sigma field of $\Bcal(\RR \times \RR_+\times \RR)\otimes \Fcal^N_{\infty}$ of stochastic processes $Z:=(Z_{(t,\theta,x)})_{(t,\theta,x) \in \RR \times \RR_+ \times \RR}$ in $\LL^1(\Omega \times \RR \times \RR_+ \times \RR,\PP\otimes \pi\otimes \vartheta)$ such that 
    $$ D_{(t,\theta,x)} Z_{(t,\theta,x)} = 0, \quad \textrm{ for a.s. } (t,\theta,x) \in \RR \times \RR_+ \times \RR.$$
\end{definition}

\begin{remark}
\label{rem:Nshift}
    Let $(t_0,\theta_0,x_0)$ in $\RR \times \RR_+ \times \RR$, $(s,t)$ in $\RR^2$ with $t_0 < s < t$.\\
    For $\mathcal T \in \{(s,t), (s,t], [s,t), [s,t]\}$ and $B$ in $\Bcal(\RR \times \RR_+ \times \RR)$,  we have that : 
    $$ N\circ \eps_{(t_0,\theta_0,x_0)}^+ (\Tcal \times B) = N (\Tcal \times B). $$
\end{remark}

\begin{definition}
\label{def: div op}
    We set $\Scal$ the set of stochastic processes $Z:=(Z_{(t,\theta,x)})_{(t,\theta,x) \in \RR\times \RR_+ \times \RR}$ in $\Ical$ such that $Z$ is predictable according to the natural history and: 
    $$ \E{\int_{\RR\times \RR_+ \times \RR} \left|Z_{(t,\theta)}\right|^2 dt d\theta \vartheta(dx)} + \E{\left(\int_{\RR\times \RR_+ \times \RR} Z_{(t,\theta)} dt d\theta \vartheta(dx) \right)^2}<+\infty .$$
    For $Z$ in $\Scal$, we set the divergence operator with respect to $N$ as  
    \begin{equation}
    \label{eq:delta}
        \delta(Z):=\int_{\RR\times \RR_+ \times \RR} Z_{(t,\theta,x)} N(dt,d\theta,dx) - \int_{\RR\times \RR_+} Z_{(t,\theta)} dt d\theta \vartheta(dx).
    \end{equation}
\end{definition}

\noindent We conclude this section with the integration by parts formula on the Poisson space (see \cite[Remark 1]{picard_formules_1996}) and the Heisenberg equality.

\begin{proposition}[See \textit{e.g.} \cite{picard_formules_1996}]
\label{prop:IPP}
    Let $F$ be in $\LL^2(\Omega,\Fcal_\infty^N,\PP)$ and $Z=(Z_{(t,\theta,x)})_{(t,\theta,x) \in \RR\times \RR_+ \times \RR}$ be in $\Scal$. We have that 
    \begin{equation}
    \label{eq:IBPPoisson}
        \E{F \delta(Z)} = \E{\int_{\RR\times \RR_+ \times \RR} Z_{(t,\theta)} D_{(t,\theta,x)} F dt d\theta \vartheta(dx)}.
    \end{equation}
\end{proposition} 

\begin{proposition}[Heisenberg commutation property]
\label{prop: heisenberg}
    For $Z\in \Scal$,
    $$D_{(t,\theta,x)}\left(\delta \left( Z\right) \right) = Z_{(t,\theta,x)}+\delta \left( D_{(t,\theta,x)}\left( Z\right)\right), \quad (t,\theta,x)\in \RR\times \RR_+ \times \RR.$$
\end{proposition}

\begin{notation}
    In this article, we may apply the Malliavin derivative to some random vector. We therefore use the same notation as in the 1-dimensional case and so for $\bx = \left(x_1, \dots, x_n \right)$ a random vector we have:
\begin{align*}
    D_{(t,\theta,x)}(\bx) := \left(D_{(t,\theta,x)}(x_1), \dots, D_{(t,\theta,x)}(x_n) \right), \quad (t,\theta,x) \in \RR \times \RR_+ \times \RR .
\end{align*}
Hence, each previous result on Malliavin's calculus holds coordinate-wise. 
\end{notation}

\subsection{Non-linear compound marked Hawkes processes}
We naturally begin this section with a formal definition of non-linear compound marked Hawkes processes with $(\Upsilon= - \infty)$ or without $(\Upsilon>-\infty)$ past history. All of this section is an extension of existing results, written in \cite{coutin_normal_2024} and \cite{khabou_sur_2022} for instance. More precisely, we couple the results of the two last mentioned works to get results for the non-linear marked Hawkes processes.
\begin{definition}[Non-linear compound marked Hawkes process]
\label{def:Hawkes}
Let $(\Omega, \FF, \PP)$ be a general probability space.
Let $\mu\in \RR$, $\phi:\RR_+ \to \RR$, $h: \RR \to \RR_+$, $b: \RR \to \RR$ and $g: \RR\to \RR_+$. Suppose that $(X_i)_{i\in \NN}$ are i.i.d random variable with distribution $\vartheta$. A non-linear compound marked Hawkes process $L^{(\Upsilon)}:=(L^{(\Upsilon)}_t)_{t \geq 0}$ with parameters $(h,\phi, b, \vartheta)$ is define by the following sum:
\begin{equation*}
    L^{(\Upsilon)}_t = \sum_{i=1}^{H^{\Upsilon}_t} g(X_i)
\end{equation*}
where $H^{(\Upsilon)}:=(H^{(\Upsilon)}_t)_{t \geq 0}$ is a counting process such that   
\begin{itemize}
\item[(i)] $H^{(\Upsilon)}_0=0,\quad \PP-a.s.$,
\item[(ii)] its ($\FF$-predictable) intensity process is given by
$$\lambda^{(\Upsilon)}_t:=h \left(\mu + \int_{(\Upsilon,t)} \phi(t-s)b\left(X_{H^{(\Upsilon)}_s} \right) dH^{(\Upsilon)}_s\right), \quad t\geq \Upsilon,$$
that is for any $\Upsilon \leq s \leq t $ and $A \in \Fcal_s$, $ \E{\1_A (H^{(\Upsilon)}_t-H^{(\Upsilon)}_s)} = \E{\int_{(s,t]} \1_A \lambda^{(\Upsilon)}_r dr }$.
\end{itemize}
\end{definition}
\begin{remark}
    Note that the process $H^{(\Upsilon)}$ is a non linear marked Hawkes process whose intensity is given by $\lambda^{(\Upsilon)}$.
\end{remark}

For the sake of this paper, we make use of the following notation:
\begin{notation}
   \begin{itemize}
    \item[$\bullet$] We will indicate by $L^{(\Upsilon)} =L^{(-\infty)} =:  L^\infty$, $H^{(\Upsilon)} =H^{(-\infty)} =:  H^\infty$ and $\lambda^{(\Upsilon)} = \lambda^{(-\infty)} =: \lambda^{\infty}$ when we deal with stationary Hawkes processes. 
    \item[$\bullet$] We will indicate by $L^{(\Upsilon)} =L^{(0)} =:  L$, $H^{(\Upsilon)} =H^{(0)} =:  H$ and $\lambda^{(\Upsilon)} = \lambda^{(0)} =: \lambda$ when we deal with Hawkes processes with empty past history. 
\end{itemize} 
\end{notation}

Inspired by \cite{bremaud_stability_1996} and \cite{khabou_sur_2022}, we introduce the following technical assumption :
\begin{assumption} \label{assump: noyau}
    We assume that: \begin{enumerate}
        \item $g(X_1)$ has a finite moment of order $4$. We make use of the following notations:
        $$m_{g,i} = \int_{\RR} \left|g(x)\right|^i \vartheta(dx),\quad i\in \{1,2,3,4\}. $$
        \item $b(X_1)$ has a finite moment of order $2$. We make use of the following notations:
        $$m_{b,i} = \int_{\RR} \left|b(x)\right|^i \vartheta(dx),\quad i\in \{1,2\}. $$
        \item $\phi : \RR_+ \to \RR$ is locally bounded and such that
        $$\left\Vert \phi\right\Vert_{\LL^1}:= \int_0^{\infty} |\phi(t)| dt < +\infty \quad \text{and} \quad m:= \int_0^{\infty} t |\phi(t) | dt < \infty$$
        \item $h: \RR \to \RR_+$ is positive and $\alpha$-Lipschitz,  with $\alpha \left\Vert \phi\right\Vert_{\LL^1} m_{b,1} < 1$.
    \end{enumerate}
\end{assumption}

\begin{remark}
    Let us discuss the different assumptions compared to the state of the art.\\ First, the existence of the fourth moment of $g(X_1)$ is not necessary in \cite{khabou_sur_2022}. Indeed, the aforementioned paper has only supposed a moment of order $3$. The difference is explained by the fact that we work in a functional setting and our methodology requires to compute the fourth moment of the quantity of interest.\\
    Secondly, we extend the convergence result that was given in \cite{zhu_nonlinear_2013} by not supposing any monotonicity on $h$ or $\phi$. Thus, we offer in this paper a new fCLTs for non linear Hawkes process with minimum assumption.
\end{remark}
The definition provided \ref{def:Hawkes} before is not the one used in this paper. Instead, we opt for representing the Hawkes process through a solution of a SDE. This representation, called the Poisson Imbedding, can be found in \cite{bremaud_stability_1996} and constitute the main tool of our analysis. Specifically, they established the following theorem and corollary:

\begin{theorem}
    \label{th:BM}
    Under Assumption \ref{assump: noyau}, the SDE below admits a unique $\FF^N$-predictable solution $\lambda^{\infty}$. 
    \begin{equation}
    \label{eq:Int_Hawkes_station}
    \lambda^{\infty}_t = h \left(\mu + \int_{(-\infty,t) \times \RR_+ \times \RR} \phi(t-u)\ind{\theta \leq \lambda_u^{\infty}} b(x) N(du,d\theta,dx) \right) ,\quad t \in \RR_+ .
    \end{equation}
    Moreover, for any $t\in \RR_+$, $\E{\lambda^\infty_t} < +\infty.$
\end{theorem}

\begin{corollary}
    Suppose Assumption \ref{assump: noyau} holds. We denote by $\lambda^{\infty}$ the unique solution of \eqref{eq:Int_Hawkes_station} and by $H^{\infty}$ his associated counting process, i.e.  
    $$H^{\infty}_t := \int_{(-\infty,t]\times \RR_+ \times \RR} \1_{\{\theta \leq \lambda^{\infty}_s\}} N(ds,d\theta,dx), \quad t\in \RR_+.$$
    Then, $H^{\infty}$ is a non linear marked Hawkes process with finite average intensity and whose intensity is $\lambda^{\infty}$ for the natural history of $H^\infty$.\\
    In addition, we can define the non linear compound marked Hawkes process $L^\infty$ by
    \begin{equation*}
        L^\infty_t := \int_{(-\infty,t]\times \RR_+ \times \RR} g(x) \1_{\{\theta \leq \lambda^{\infty}_s\}} N(ds,d\theta,dx), \quad t\in \RR_+.
    \end{equation*}
\end{corollary}
Note that the previous theorem and corollary deal with Hawkes processes on the entire real line. In our case, we work with Hawkes processes that are said to have an empty history i.e. $L(\RR_-)= H(\RR_-)=\emptyset$. We will not give the proof of the following result since it is an adaptation of Theorem 4.1.7 \cite{coutin_normal_2024}, for the non linearity of $\lambda$, combined with the Theorem 3.3 of \cite{hillairet_malliavin-stein_2022}, for the use of a tri-dimensional Poisson random measure.
\begin{theorem}
    \label{th:HRR}
    Under Assumption \ref{assump: noyau}, the SDE below admits a unique $\FF^N$-predictable solution $\lambda$. 
    \begin{equation}
    \label{eq:Int_Hawkes}
    \lambda_t = h \left(\mu + \int_{(0,t)\times \RR_+ \times \RR} \phi(t-u)\ind{\theta \leq \lambda_u} b(x)N(du,d\theta,dx) \right) ,\quad t \in \RR_+ .
    \end{equation}
    Moreover, for any $t\in \RR_+$, $\E{\lambda_t} < +\infty.$
\end{theorem}

\begin{corollary}
    Suppose that Assumption \ref{assump: noyau} holds. We consider $(L,H,\lambda)$ the unique solution to the SDE 
    \begin{equation}
    \left\lbrace
    \begin{array}{l}
        L_t = \int_{(0,t] \times \RR_+ \times \RR} \ind{\theta \leq \lambda_s} g(x)N(ds,d\theta,dx),\quad t \in \RR_+,\\[2ex]
        H_t = \int_{(0,t] \times \RR_+ \times \RR} \ind{\theta \leq \lambda_s} N(ds,d\theta,dx),\quad t \in \RR_+, \\[2ex]
        \lambda_t = h \left(\mu + \int_{(0,t)\times \RR_+ \times \RR} \phi(t-u)\ind{\theta \leq \lambda_u} b(x)N(du,d\theta,dx) \right),\quad t \in \RR_+
    \end{array}
    \right. .
    \end{equation}
    We set $\FF^L$ the natural filtration of $L$. Then $L$ is a Hawkes process in the sense of Definition \ref{def:Hawkes}.
\end{corollary}

\subsection{Elements of Malliavin calculus  for non-linear compound marked Hawkes processes in continuous time}
\label{subsec: Malliavin}
Malliavin's calculus for Hawkes processes have been largely studied over the past two decades. We then recall what we estimate to be the most important result of Malliavin's calculus for the presented work. In the same line as for the Poisson imbedding, we do not give the proof since it is the same scheme of proof as presented in Chapter 3 of \cite{khabou_sur_2022}. \\
In this part, we decide to use a non-linear Hawkes process with empty history. Note that the results can also be use for a stationary Hawkes process, that is why we use the same notation as Definition \ref{def:Hawkes}.
%%%%%%%%%%%%%%%%%%%%

\begin{lemma}
    Let $t \geq \Upsilon$ and $(\theta, \theta_0) \in \RR^2_+$, it holds that :
    \begin{align*}
    &\ind{\theta \leq \lambda^{(\Upsilon)}_t}\ind{\theta_0 \leq \lambda^{(\Upsilon)}_t} \left(L^{(\Upsilon)}_s \circ \eps_{(t,\theta,x)}^+, H^{(\Upsilon)}_s \circ \eps_{(t,\theta,x)}^+ , \lambda^{(\Upsilon)}_s\circ \eps_{(t,\theta,x)}^+ \right)_{s\geq \Upsilon} \\
    &=\ind{\theta \leq \lambda^{(\Upsilon)}_t}\ind{\theta_0 \leq \lambda^{(\Upsilon)}_t} \left( L^{(\Upsilon)}_s \circ \eps_{(t,\theta_0,x)}^+, H^{(\Upsilon)}_s \circ \eps_{(t,\theta_0,x)}^+ , \lambda^{(\Upsilon)}_s\circ \eps_{(t,\theta_0,x)}^+ \right)_{s\geq \Upsilon} .
    \end{align*}
\end{lemma}

This Lemma thereupon yields Proposition \ref{prop: non importance du saut}:
\begin{proposition}
    \label{prop: non importance du saut} 
    Let $F$ be an $\Fcal_\infty^L$-measurable random variable. Then for any $t\geq \Upsilon$, for any $(\theta, \theta_0) \in \RR^2_+$ and for any $x\in \RR$,
    $$\ind{\theta \leq \lambda^{(\Upsilon)}_t} D_{(t,\theta,x)} \left( F \right) = \ind{\theta_0 \leq \lambda^{(\Upsilon)}_t} D_{(t,\theta_0,x)} \left( F \right)$$
\end{proposition}

\begin{theorem}[Integration by parts]
\label{th:IPPH}
Set $\z:=(\z_{(t,\theta)})_{(t,\theta) \in (\Upsilon, +\infty)\times\RR_+}$ the stochastic process defined as 
$$ \z_{(t,\theta)}:= \ind{\theta \leq \lambda^{(\Upsilon)}_t}, \quad (t,\theta) \in (\Upsilon, +\infty)\times\RR_+.$$
Let $Z:=(Z_{t,x})_{t \geq \Upsilon, x\in \RR}$ be a $\FF^H$-predictable process satisfying
$$\E{\int_\RR  \int_\Upsilon^{+\infty} |Z_{(t,x)}|^2 \lambda^{(\Upsilon)}_t dt \vartheta(dx) + \left(\int_\RR  \int_\Upsilon^{+\infty} Z_{(t,x)} \lambda^{(\Upsilon)}_t dt \vartheta(dx) \right)^2}<\infty.$$ It holds that 
\begin{itemize}
\item[(i)] $Z \z=\left(Z_{(t,x)} \ind{\theta \leq \lambda^{(\Upsilon)}_t}\right)_{(t,\theta,x)\in (\Upsilon, +\infty)\times\RR_+\times \RR}$ belongs to $\Scal$.
\item[(ii)] For any $\Fcal^H_\infty$-measurable random variable $F$ with $\E{|F|^2}<+\infty$,
\begin{equation}
\label{eq:IBP}
\E{F \delta(Z \ind{\theta \leq \lambda^{(\Upsilon)}_t})} = \E{ \int_\RR  \int_\Upsilon^{+\infty} \lambda^{(\Upsilon)}_t Z_{(t,x)} D_{(t,0,x)} F dt \vartheta(dx)}.
\end{equation}
\end{itemize}
\end{theorem}

 %%%%%%%%%%%%%%%%%%%%%%%%%%
\section{Main result}
\label{sec: main result}
The aim of this paper is to give the first fCLT regarding non-linear compound marked Hawkes process. Our methodology also provide the first upper bound of the convergence rate. To do so, we will make use the same notation as before and denote the non-linear compound marked Hawkes process $L$ and its ground process $H$ (which is a non-linear marked Hawkes process whose intensity is $\lambda$). This result will provide the first functional central limit theorem. Besides, we improve the result obtained by Zhu \cite{zhu_nonlinear_2013} by reducing the monotony-type assumptions on the kernel $\phi$ and $h$. In particular, we give the first upper bound of the functional Wasserstein distance between the normalized martingale $F^{(T)}$ and a Brownian motion of variance $\sigma^2 m_{g,2} =: \Tilde{\sigma}^2$ with $\sigma^2>0$ being the expectation of the intensity of the marked non-linear Hawkes process with past history $\lambda^\infty$ and $m_{g,2} = \int_{\RR} \left|g(x) \right|^2 \vartheta(dx)$. Formally, we are interested in the following:
$$d_W \left( F^{(T)},\sigma \sqrt{m_ {g,2}}B \right) = \sup_{f\in \left(\lip_1^{\DD}, \left\|\cdot\right\|_{\LL^\infty(0,1)} \right)} \E{f\left(F^{(T)}\right)} - \E{f(B)},$$
where $\lip^{\DD}_1$ is the set of 1-Lipschtiz functions from $\DD([0,1],\RR)$ to $\RR$, $B$ a standard Brownian motion and where $F^{(T)}$ is defined by
\begin{align}
    F^{(T)}_t := F_{tT} = \frac{L_{Tt}- m_{g,1}\int_0^{Tt}\lambda_s ds}{\sqrt{T}}, \quad t\in [0,1].
\end{align}
with $m_{g,1} = \int_{\RR} g(x) \vartheta(dx)$.\\
We now state our main theorem:
\begin{theorem}
\label{thm: main}
Under Assumption \ref{assump: noyau}, we have that $\left(\frac{L_{tT}- m_{g,1}\int_0^{tT} \lambda_s ds}{\sqrt{T}}\right)_{t\in[0,1]}$ converges in law for the Skorokhod topology to $\left( \sigma \sqrt{m_{g,2}} B_t\right)_{t\in [0,1]}$ as $T\to +\infty$ where $B$ stands for the standard Brownian motion.
Moreover, there exists $C>0$ such that for any $T>0$,
    \begin{align*}
        d_W \left( F^{(T)}, \Tilde{\sigma}B \right) \leq C\frac{\ln(T)}{T^{1/10}}.
    \end{align*}
\end{theorem}
\begin{remark}
\begin{enumerate}
    \item This theorem provides the first convergence rate result for non-linear Hawkes processes. In particular, by setting $b \equiv 1$ and $g\equiv 1$, we recover the classical non-linear Hawkes process. Moreover, this result extends the existing convergence results in the literature. In particular, in \cite{zhu_nonlinear_2013}, stronger assumptions are imposed on the kernel function $\phi$ and the nonlinear function $h$: $\phi$ is required to be decreasing and $h$ increasing. In Theorem \ref{thm: main}, by contrast, we make no assumption regarding the monotonicity of these functions.
    \item Let us discuss on the convergence rate itself. In Section 4.1, \cite{fomichov_implementable_2021} establishes that the 2-Wasserstein distance between a rescaled Poisson process and a Brownian motion is of order $T^{-1/4}\log(T)$. Hence, achieving a better rate appears to be challenging. Besides, \cite{besancon_diffusive_2024} whose work focus on linear Hawkes process quantify the convergence rate between $\Theta_T(H)$ and $\sigma \Theta_T(B)$ (see Lemma 6.1, \cite{besancon_diffusive_2024}) where $\Theta_T(f)$ denotes the linear interpolation of $f$ along the subdivision $\left(i/T\right)_{i=0}^T$ of $[0,1]$. The convergence rate obtained is of order $T^{-1/6}\log(T)$ in 1-Wasserstein distance. However, the error between the Hawkes process and its linear interpolation is not addressed there; an error that we control in our analysis and which significantly reduces the overall convergence rate.
\end{enumerate}
\end{remark}

To prove Theorem \ref{thm: main} , we follow the methodology of \cite{coutin_donskers_2020} and start by divided the quantity in three other quantities containing piecewise constant approximation of our processes. We define the regular subdivision $(t_i)_{i=0, \dots, n}$ of $[0,1]$ by 
\begin{align}
\label{def: subdivision}
    \forall i=0, \dots,n,\quad t_i := \frac{i}{n}.
\end{align}

For $n\in \NN^*$ and a process $h$, we thereupon set the piecewise constant function 
$$\Pi_n (h)_t := \sum_{i=0}^{n-1} h(t_i)\1_{[t_i, t_{i+1})}(t) + h(t_n)\1_{\{t_n\}}(t).$$
With this notation on hand and by denoting $\Tilde{\sigma}^2 := \sigma^2 m_{g,2}$, we can write
\begin{align*}
    \sup_{f\in \left(\lip_1^{\DD}, \left\|\cdot\right\|_{\LL^\infty(0,1)}\right)}\E{f\left( F^{(T)}\right)}-\E{f\left( \Tilde{\sigma}B\right)} \leq& \sup_{f\in \left(\lip_1^{\DD}, \left\|\cdot\right\|_{\LL^\infty(0,1)}\right)}\left\{ \E{f\left( F^{(T)}\right)}-\E{f\left( \Pi_n\left( F^{(T)} \right)\right)} \right\}\\
    &+ \sup_{f\in \left(\lip_1^{\DD}, \left\|\cdot\right\|_{\LL^\infty(0,1)}\right)}\left\{\E{f\left( \Pi_n\left( F^{(T)} \right)\right)}-\E{f\left( \Pi_n\left( \Tilde{\sigma}B \right)\right)} \right\}\\
    &+ \sup_{f\in \left(\lip_1^{\DD}, \left\|\cdot\right\|_{\LL^\infty(0,1)}\right)}\left\{ \E{f\left( \Pi_n\left( \Tilde{\sigma}B \right)\right)} - \E{f\left( \Tilde{\sigma}B\right)} \right\}.
\end{align*}

Whereas the proof of Theorem \ref{thm: main} consists in looking for upper bound of these three quantities, the main part of the proof relies on the control of the second term. We decide to write its upper bound in the following theorem: 

\begin{theorem}
\label{thm: maj_PiF_PiB}
    Under Assumption \ref{assump: noyau}, there exists $C>0$ such that for any $T>0$ and any $n\in \NN^*$, 
    \begin{align*}
        \sup_{f\in \left(\lip_1^{\DD}, \left\|\cdot\right\|_{\LL^\infty(0,1)}\right)}\left\{\E{f\left( \Pi_n\left( F^{(T)} \right)\right)}-\E{f\left( \Pi_n\left( \Tilde{\sigma}B \right)\right)} \right\} \leq C\left( \frac{n^2}{T} + \frac{n}{\sqrt{T}} + \frac{n}{\sqrt{T}} \ln\left( \frac{T}{n} \right) +\frac{\sqrt{n}}{T}\right).
    \end{align*}
\end{theorem}
Let us give the proof of Theorem \ref{thm: maj_PiF_PiB} before giving the proof of our main result.
\begin{proof}[Proof of Theorem \ref{thm: maj_PiF_PiB}]
   The proof is divided in two steps. The first step consists in reducing the problem to the control of the Wasserstein distance between two finite-dimensional vectors: one formed by a finite number of increments of the Hawkes process, and the other by the corresponding finite number of increments of the Brownian motion. We then apply the Stein–Malliavin method to derive an upper bound on this distance. The main difference with \cite{coutin_normal_2024} lies in the fact that the present argument is carried out in a genuinely vector-valued setting, which requires controlling several quantities depending on the dimension, that is, on the number $n$ of increments involved.\\
   
    Let us recall that $\Tilde{\sigma}^2 = \sigma^2 m_{g,2}$ where $\sigma^2 = \E{\lambda^\infty_0}$ and $m_{g,2} = \int_\RR \left|g(x)\right|^2 \vartheta(dx)$.
\paragraph{Step 1:}~\\
The goal of this part is to find an upper bound that no more depends on $\Pi_n(F^{(T)})$ and  $\Tilde{\sigma}\Pi_n(B)$ but depends on $\bdF^{(T)}$ and $\bG$ where 
\begin{align}
\label{def: dF^T_and_G}
    \bdF^{(T)} := \left(F^{(T)}_{t_1}-F^{(T)}_{t_0}, \dots, F^{(T)}_{t_n}-F^{(T)}_{t_{n-1}} \right) \quad \text{and}\quad \bG := \Tilde{\sigma}\left( B_{t_1}-B_{t_0}, \dots, B_{t_n}-B_{t_{n-1}}\right).
\end{align}
Note that $\bG$ is a centered Gaussian vector with covariance matrix $\bsig = \frac{\Tilde{\sigma}^2}{n}\mathbf{I_n}$.
To do so, we set the function $\chi: \RR^n \to \DD\left( [0,1],\RR\right)$ which is defined by 
\begin{align*}
    \chi: (x_i)_{i=0,\dots,n-1} \mapsto \sum_{i=0}^{n-1} \left(\sum_{j=0}^{i} x_j\right) \1_{[t_i, t_{i+1})} + \sum_{j=0}^{n-1} x_j \1_{\{t_n\}}.
\end{align*}
Then, for any $f\in \left(\lip_1^{\DD}, \left\| \cdot \right\|_{\LL^\infty(0,1)} \right)$, we have 
\begin{align*}
    f\circ \Pi_n\left(F^{(T)} \right) = f\circ \chi \left(\bdF^{(T)} \right) \quad \text{and}\quad f\circ \Pi_n\left(\Tilde{\sigma} B \right) = f\circ \chi \left(\bG\right).
\end{align*}
Moreover, it can be shown that $f\circ \chi \in \left(\lip_1^{\RR^n}, \left\| \cdot\right\|_{\infty,1} \right)$ with $\left\| \cdot\right\|_{\infty,1}$ defined by
\begin{equation*}
    \forall \mathbf{x}= (x_0,\dots, x_{n-1})\in \RR^n, \left\|\mathbf{x}\right\|_{\infty,1} := \max_{i=0,\dots, n-1}\left| \sum_{j=0}^{i} x_j\right|.
\end{equation*}
And so, for any $\mathbf{x}\in \RR^n$, since $\left\| \chi (\mathbf{x})\right\|_{\LL^{\infty}(0,1)}= \max_{i=0,\dots, n-1}\left| \sum_{j=0}^{i}x_j\right| =:\left\|  \mathbf{x} \right\|_{\infty,1}$, we have 
\begin{align*}
    \sup_{f\in \left(\lip_1^{\DD}, \left\| \cdot \right\|_{\LL^\infty(0,1)}\right)} \E{f\left( \Pi_n (F^{(T)})\right)} - \E{f\left( \Pi_n(\Tilde{\sigma}B)\right)} \leq \sup_{f\in \left(\lip_1^{\RR^n}, \left\| \cdot \right\|_{\infty,1}\right)} \E{f\left( \bdF^{(T)} \right)} - \E{f\left( \bG \right)}.
\end{align*}
From now, let $f\in \left(\lip_1^{\RR^n}, \left\| \cdot\right\|_{\infty,1} \right)$ be fixed.
Then, for any $\eps>0$, there exists a $\Ccal^2$ approximation of $f$, denoted by $f_\eps$, which respects the following properties
\begin{align}
\label{eq: prop_f_eps}
    &\lim_{\eps\to 0^+} \left\|f-f_\eps \right\|_{\LL^\infty(\RR^n)} =0, & \max_{i=1, \dots, n}\left\|\partial_i f_\eps \right\|_{\LL^\infty(\RR^n)} \leq \left\| f_\eps\right\|_{lip} \leq \left\| f \right\|_{lip} \leq 1.
\end{align}
For more information on this kind of approximation we refer to \cite{nourdin_multivariate_2010} or to \cite{tudor_multidimensional_2023}.\\
So, by the triangular inequality, it holds that
\begin{align}
\label{eq: Lip_to_C2}
    \E{f(\bdF^{(T)})}-\E{f(\bG)} \leq& 2\left\| f-f_{\eps}\right\|_{\LL^\infty(\RR^n)} + \E{f_{\eps}(\bdF^{(T)})} - \E{f_{\eps}(\bG)}.
\end{align}
This inequality allows us to work with function that are 1-Lipschitz and continuously differentiable with bounded derivatives.

\paragraph{Step 2: Application of the Malliavin-Stein method}~\\
This part of the proof begins with a classical result due to the Stein method. More precisely, let us introduce the Ornstein-Uhlenbeck semigroup 
\begin{align*}
    P_t f_\eps : \bx\in \RR^n \mapsto \E{f_\eps\left( e^{-t}\bx + \sqrt{1-e^{-2t}}\bG\right)}, \quad t\geq 0.
\end{align*}
Then, we have the following equality:
\begin{align}
\label{eq: Stein_f_eps}
    f_\eps(\bx)- \E{f_\eps(\bG)} = \int_0^{\infty} \left\langle \bx, \nabla P_t f_\eps(\bx)\right\rangle_{\RR^n}-\left\langle \bsig, \nabla^2 P_t f_\eps(\bx) \right\rangle_{H.S.} dt .
\end{align}
We refer to Section \ref{subsec: Ornstein_semigroup} for the proof of this equality and result on the Ornstein-Uhlenbeck semigroup.\\
Moreover, for $\tau>0$, we also have
\begin{align}
\label{eq: Stein_tau}
    f_\eps(\bx)- \E{f_\eps(\bG)} = f_\eps(\bx)- \E{P_\tau f_\eps \left(\bdF^{(T)} \right)} + \int_\tau^{\infty} \left\langle \bx, \nabla P_t f_\eps(\bx)\right\rangle_{\RR^n}-\left\langle \bsig, \nabla^2 P_t f_\eps(\bx) \right\rangle_{H.S.} dt .
\end{align}
Therefore, plugging \eqref{eq: Stein_tau} in \eqref{eq: Lip_to_C2}, we get for any $\tau>0$, 
\begin{align*}
    \E{f(\bdF^{(T)})}-\E{f(\bG)} \leq& 2\left\| f-f_{\eps}\right\|_{{\LL^\infty(\RR^n)}} + \E{f_{\eps}(\bdF^{(T)})} - \E{P_\tau f_\eps(\bdF^{(T)})}\\
    &+\int_\tau^{\infty} \E{\left\langle \bdF^{(T)}, \nabla P_t f_\eps(\bdF^{(T)})\right\rangle_{\RR^n}-\left\langle \bsig, \nabla^2 P_t f_\eps(\bdF^{(T)}) \right\rangle_{H.S.} } dt.
\end{align*}
For the rest of the proof, we introduce the following notations:
\begin{align*}
    A^T_{\tau,\eps} &:= \E{f_{\eps}(\bdF^{(T)})} - \E{P_\tau f_\eps(\bdF^{(T)})};\\
    B^T_{\tau,\eps}(t) &:=\E{\left\langle \bdF^{(T)}, \nabla P_t f_\eps(\bdF^{(T)})\right\rangle_{\RR^n}-\left\langle \bsig, \nabla^2 P_t f_\eps(\bdF^{(T)}) \right\rangle_{H.S.} }, \quad t\in [\tau, +\infty).
\end{align*}
Since $f_\eps \in \left( \lip_1^{\RR^n}, \left\| \cdot \right\|_{\infty,1}\right)$, we obtain
\begin{align}
\label{eq: maj_A_tau_eps}
    \left|A^T_{\tau,\eps} \right| \leq (1-e^{-\tau}) \E{\left\|\bdF^{(T)} \right\|_{\infty,1}} + \sqrt{1-e^{-2\tau}} \E{\left\|\bG \right\|_{\infty,1}}.
\end{align}
Besides, we have
\begin{align*}
    \E{\left\|\bdF^{(T)} \right\|_{\infty,1}} &= \E{\max_{i=0,\dots,n} \left|\sum_{j=0}^{i-1}F^{(T)}_{t_{j+1}}-F^{(T)}_{t_{j}} \right|} = \E{\max_{i=0,\dots,n} \left|F^{(T)}_{t_{i}}\right|} \leq \sqrt{\E{\max_{i=0,\dots,n} \left|F^{(T)}_{t_{i}}\right|^2}} .
\end{align*}
And so, using the fact that $\left(F_{t_i}\right)_{i=0,\dots,n}$ is a martingale, Doob's inequality yields
\begin{align*}
    \E{\left\|\bdF^{(T)} \right\|_{\infty,1}} &\leq C \sqrt{\E{\left|F^{(T)}_{t_n}\right|^2}} = C \sqrt{\frac{m_{g,2}}{T}\int_0^T \E{\lambda_s}ds }.
\end{align*}
A direct application of Lemma $\ref{lem: maj_esp_lambda}$ yields
\begin{align*}
    \E{\left\|\bdF^{(T)} \right\|_{\infty,1}} \leq C \sqrt{\frac{h(\mu) \alpha m_{g,2} m_{b,1} \left\|\phi \right\|_{\LL^1}}{1-\alpha m_{b,1} \left\|\phi \right\|_{\LL^1}}}.
\end{align*}
Similarly, we have $\E{\left\| \bG \right\|_{\infty,1}} \leq \sigma \sqrt{m_{g,2}}$. Plugging the two last inequalities in \eqref{eq: maj_A_tau_eps}, we obtain
\begin{align}
    \left|A^T_{\tau,\eps} \right| \leq  C\sqrt{\frac{h(\mu) \alpha m_{g,2} m_{b,1} \left\|\phi \right\|_{\LL^1}}{1-\alpha m_{b,1}\left\|\phi \right\|_{\LL^1}}} (1-e^{-\tau})+ \sigma \sqrt{m_{g,2}}\sqrt{1-e^{-2\tau}} .
\end{align}
The control of second term, $B^T_{\tau,\eps}(t)$, relies on the Malliavin calculus and the different properties of the Ornstein-Uhlenbeck semigroup. The details of such method are postponed in Proposition \ref{prop: malliavin_stein_calcul} which allows us to write 
\begin{align*}
    B^T_{\tau,\eps}(t) \leq& \frac{1}{2}\max_{i=1,\dots,n}\left\|\partial_{iii}^3 P_t f_\eps \right\|_{\infty} \frac{m_{g,3}}{T^{3/2}} \int_0^T \E{\lambda_s}ds\\
    &+ m_{g,2}\max_{i=1,\dots,n}\left\|\partial_{ii}^2 P_t f_\eps \right\|_{\infty} \sum_{i=1}^n \E{\left|\frac{\sigma^2}{n}-\frac{1}{T}\int_{t_{i-1}T}^{t_i T} \lambda_s ds\right|}.
\end{align*}
By Lemma \ref{lem: maj_derive_semigroup}, we have
\begin{align*}
    B^T_{\tau,\eps}(t) \leq& \frac{m_{g,3}}{2m_{g,2}\sigma^2} \frac{e^{-5t/2}}{1-e^{-t}} \frac{n}{T^{3/2}} \int_0^T \E{\lambda_s}ds + \sqrt{m_{g,2}}\frac{\sqrt{n}e^{-3t/2}}{\sigma\sqrt{1-e^{-2t}}} \sum_{i=1}^n \E{\left|\frac{\sigma^2}{n}-\frac{1}{T}\int_{t_{i-1}T}^{t_i T} \lambda_s ds\right|}.
\end{align*}
In addition, combining Lemma \ref{lem: maj_esp_lambda} and Lemma \ref{lem: maj_sum_diff_variance}, we get
\begin{align}
\label{eq: maj_B_tau_eps}
    B^T_{\tau,\eps}(t) \leq& C\frac{n}{\sqrt{T}} \frac{e^{-5t/2}}{1-e^{-t}}+ C \frac{e^{-3t/2}}{\sqrt{1-e^{-2t}}} \left(  \frac{\sqrt{n}}{T}+\frac{n}{\sqrt{T}} \right).
\end{align}
Here, $C>0$ is a constant that depends only on $\alpha$, $\left\|\phi \right\|_{\LL^1}$, $h(\mu)$, $\sigma^2$, $m=\int_{0}^{\infty}t \phi (t) dt$ and of the $m_{b,i}$ and $m_{g,k}$ for $i=1,2$ and $k=1,2,3,4$.\\
Therefore, combining \eqref{eq: maj_A_tau_eps} and \eqref{eq: maj_B_tau_eps}, we have that there exists $C>0$ such that for any $\tau>0$, $T>0$, $\eps>0$ and $n\in \NN^\ast$,
\begin{align*}
    \E{f(\bdF^{(T)})}-\E{f(\bG)} \leq& 2\left\| f-f_{\eps}\right\|_{{\LL^\infty(\RR^n)}} + C(1-e^{-\tau}) + C\sqrt{1-e^{-2\tau}} .\\
    &+C\frac{n}{\sqrt{T}} \int_\tau^{\infty} \frac{e^{-5t/2}}{1-e^{-t}}dt + C \left( \frac{\sqrt{n}}{T}+\frac{n}{\sqrt{T}} \right)\int_\tau^{\infty}\frac{e^{-3t/2}}{\sqrt{1-e^{-2t}}}dt .
\end{align*}
Letting $\eps \to 0^+$, we obtain 
\begin{align*}
    \E{f(\bdF^{(T)})}-\E{f(\bG)} \leq& C(1-e^{-\tau}) + C\sqrt{1-e^{-2\tau}} \\
    &+C\frac{n}{\sqrt{T}} \int_\tau^{\infty} \frac{e^{-5t/2}}{1-e^{-t}}dt + C \left( \frac{\sqrt{n}}{T}+\frac{n}{\sqrt{T}} \right)\int_\tau^{\infty}\frac{e^{-3t/2}}{\sqrt{1-e^{-2t}}}dt .
\end{align*}
Moreover, we have the following upper bounds on the integrals:
\begin{align*}
    \int_\tau^{\infty} \frac{e^{-5t/2}}{1-e^{-t}}dt \leq -\ln\left( 1-e^{-\tau}\right)\quad \text{and}\quad \int_\tau^{\infty}\frac{e^{-3t/2}}{\sqrt{1-e^{-2t}}}dt \leq C.
\end{align*}
Finally, we get:
\begin{align*}
    \E{f(\bdF^{(T)})}-\E{f(\bG)} \leq& C(1-e^{-\tau})+ C\sqrt{1-e^{-2\tau}}+C\frac{n}{\sqrt{T}} \left[ -\ln\left( 1-e^{-\tau}\right)\right]+ C \left(  \frac{\sqrt{n}}{T}+\frac{n}{\sqrt{T}}\right).
\end{align*}
Taking $\tau$ such that $1-e^{-\tau}= \frac{n^2}{T}$ we conclude that
\begin{align*}
    \E{f(\bdF^{(T)})}-\E{f(\bG)} \leq& C\left( \frac{n^2}{T} + \frac{n}{\sqrt{T}} + \frac{n}{\sqrt{T}} \ln\left( \frac{T}{n} \right) +\frac{\sqrt{n}}{T}\right).
\end{align*}
\end{proof}

We now give the proof of Theorem \ref{thm: main}.
\begin{proof}[Proof of Theorem \ref{thm: main}]
    The proof of this theorem consists of controlling three quantities given by the following inequality:
    \begin{align*}
    \sup_{f\in \left(\lip_1^{\DD}, \left\|\cdot\right\|_{\LL^\infty(0,1)}\right)}\E{f\left( F^{(T)}\right)}-\E{f\left( \Tilde{\sigma}B\right)} \leq& \sup_{f\in \left(\lip_1^{\DD}, \left\|\cdot\right\|_{\LL^\infty(0,1)}\right)}\left\{ \E{f\left( F^{(T)}\right)}-\E{f\left( \Pi_n\left( F^{(T)} \right)\right)} \right\}\\
    &+ \sup_{f\in \left(\lip_1^{\DD}, \left\|\cdot\right\|_{\LL^\infty(0,1)}\right)}\left\{\E{f\left( \Pi_n\left( F^{(T)} \right)\right)}-\E{f\left( \Pi_n\left( \Tilde{\sigma}B \right)\right)} \right\}\\
    &+ \sup_{f\in \left(\lip_1^{\DD}, \left\|\cdot\right\|_{\LL^\infty(0,1)}\right)}\left\{ \E{f\left( \Pi_n\left( \Tilde{\sigma}B \right)\right)} - \E{f\left( \Tilde{\sigma}B\right)} \right\}.
\end{align*}
The upper bound regarding the second term is given by Theorem \ref{thm: maj_PiF_PiB}. Let us focus on the first term. For every $f\in \left(\lip_1^{\DD}, \left\|\cdot\right\|_{\LL^\infty(0,1)}\right)$, we have:
\begin{align*}
    \E{f\left( F^{(T)}\right)}-\E{f\left( \Pi_n\left( F^{(T)} \right)\right)} \leq \E{\left\| F^{(T)} - \Pi_n\left( F^{(T)} \right)\right\|_{\LL^\infty(0,1)}} \leq \E{\sup_{t\in [0,1]} \left| F^{(T)}_t - \Pi_n\left( F^{(T)} \right)_t\right|^4}^{1/4}.
\end{align*}
Using the discretization $(t_i)_{i=0,\dots,n}$, we split the supremum over $[0,1]$ to get suprema on smaller intervals as follows:
\begin{align*}
    \E{\sup_{t\in [0,1]} \left| F^{(T)}_t - \Pi_n\left( F^{(T)} \right)_t\right|^4} \leq \E{\sum_{i=0}^{n-1} \sup_{t\in [t_i, t_{i+1})} \left|F^{(T)}_t - F^{(T)}_{t_i}\right|^4}.
\end{align*}
Since $\left(F^{(T)}_{t_i}-F^{(T)}_t\right)_{t\in [t_i, t_{i+1})}$ is a martingale, Doob's inequality yields
\begin{align*}
    \left|\E{f\left( \Pi_n\left( F^{(T)} \right)\right)} - \E{f\left( F^{(T)}\right)} \right| \leq c_4 \left(\E{\sum_{i=0}^{n-1} \left|F^{(T)}_{t_i}-F^{(T)}_{t_{i+1}} \right|^4 }\right)^{1/4},
\end{align*}
where $c_4>0$ is an universal constant. By the Burkholder-Davis-Gundy inequality, we get
\begin{align*}
    \left|\E{f\left( \Pi_n\left( F^{(T)} \right)\right)} - \E{f\left( F^{(T)}\right)} \right| \leq \bar{c}_4 \left(\frac{m_{g,2}^2}{T^2} \sum_{i=0}^{n-1}\E{\left(\int_{t_i T }^{t_{i+1}T} \lambda_t dt\right)^2+m_{g,4} \int_{t_iT}^{t_{i+1}T }\lambda_tdt } \right)^{1/4}
\end{align*}
where $\bar{c}_4>0$ is an universal constant.
Then, by Lemma \ref{lem: maj_esp_lambda} and Lemma \ref{lem: maj_sum_int_lambda_carré}, there exists a constant $C>0$, depending only on $\alpha$, $\left\|\phi \right\|_{\LL^1}$, $h(\mu)$, $\sigma^2$, $m=\int_{0}^{\infty}t \phi (t) dt$ and the two first $b$-moments and the second and fourth $g$-moments of $\vartheta$ denoted by $m_{b,i}$ and $m_{g,i}$, such that for any $T>0$ and any $n\in \NN^\ast$,
\begin{align}
\label{eq: maj_F_PiF}
    \E{f\left( F^{(T)}\right)}-\E{f\left( \Pi_n\left( F^{(T)} \right)\right)} \leq C\left(\frac{1}{T}+\frac{1}{n} \right)^{1/4}.
\end{align}
On the other hand, we can use the same methodology to get:
\begin{align*}
    \E{f\left( \Tilde{\sigma}B\right)}-\E{f\left( \Pi_n\left( \Tilde{\sigma}B\right)\right)} \leq c_4 \left(\E{\sum_{i=0}^{n-1} \left|B_{t_i}-B_{t_{i+1}} \right|^4 }\right)^{1/4}.
\end{align*}
Since for any $i=0,\dots, n-1$, $B_{t_i}-B_{t_{i+1}} \sim \Ncal\left(0, \frac{\Tilde{\sigma}}{n} \right)$ with $\Tilde{\sigma}^2 = \sigma^2 m_{g,2}$, we obtain:
\begin{align}
\label{eq: maj_B_PiB}
    \E{f\left( \Tilde{\sigma}B\right)}-\E{f\left( \Pi_n\left( \Tilde{\sigma}B\right)\right)} \leq 3c_4 \left(\sum_{i=0}^{n-1} \frac{m_{g,2}^2\sigma^4}{n^2} \right)^{1/4}\leq 3c_4 \frac{\sigma \sqrt{m_{g,2}} }{n^{1/4}}.
\end{align}
Combining \eqref{eq: maj_F_PiF} and \eqref{eq: maj_B_PiB} with Theorem \ref{thm: maj_PiF_PiB}, we have that there exists $C>0$ depending only on $\alpha$, $\left\|\phi \right\|_{\LL^1}$, $h(\mu)$, $\sigma^2$, $m=\int_{0}^{\infty}t \phi (t) dt$ and the $b$-moments and $g$-moments of $\vartheta$, such that for any $T>0$ and any $n\in \NN^\ast$,
\begin{align*}
    \sup_{f\in \left(\lip_1^{\DD}, \left\|\cdot\right\|_{\LL^\infty(0,1)}\right)}\E{f\left( F^{(T)}\right)}-\E{f\left( \Tilde{\sigma}B\right)} \leq C\left( \frac{1}{T^{1/4}}+ \frac{1}{n^{1/4}} + \frac{n^2}{T} + \frac{n}{\sqrt{T}} + \frac{n}{\sqrt{T}} \ln\left( \frac{T}{n} \right) +\frac{\sqrt{n}}{T}\right).
\end{align*}
Finally, taking $n = \lfloor T^{2/5} \rfloor +1$ we get
\begin{align*}
    \sup_{f\in \left(\lip_1^{\DD}, \left\|\cdot\right\|_{\LL^\infty(0,1)}\right)}\E{f\left( F^{(T)}\right)}-\E{f\left( \Tilde{\sigma}B\right)} \leq C\frac{\ln(T)}{T^{1/10}}.
\end{align*}
\end{proof}

\section{Technical Lemmata}
\label{sec: tech_lem}
\subsection{Result on the non-linear (marked) Hawkes process}

In this section, we establish several results concerning the expectation of the intensity of the non linear marked Hawkes process that we denote by $H$. Although these results do not appear to be entirely new, to the best of our knowledge, there is no existing work addressing both non-linearity and marks simultaneously. Moreover, we provide inequalities that can be applied in the context of compound marked Hawkes process, denoted by $L$. Even though the proofs are rather standard and can be found, for instance, in \cite{coutin_normal_2024} or \cite{khabou_sur_2022}, we include them here for the sake of completeness.
Recall that our different processes $L$, $H$ and $\lambda$ respect the following system:
\begin{equation*}
    \left\lbrace
    \begin{array}{l}
        L_t = \int_{(0,t] \times \RR_+ \times \RR} \ind{\theta \leq \lambda_s} g(x)N(ds,d\theta,dx),\quad t \in \RR_+,\\[2ex]
        H_t = \int_{(0,t] \times \RR_+ \times \RR} \ind{\theta \leq \lambda_s} N(ds,d\theta,dx),\quad t \in \RR_+, \\[2ex]
        \lambda_t = h \left(\mu + \int_{(0,t)\times \RR_+ \times \RR} \phi(t-u)\ind{\theta \leq \lambda_u} b(x)N(du,d\theta,dx) \right),\quad t \in \RR_+
    \end{array}
    \right. .
    \end{equation*}
\begin{lemma}
\label{lem: maj_esp_lambda}
     For all $s\in \RR_+$, 
    \begin{align*}
        \E{\lambda_s} \leq h(\mu)\left(1+ \left\| \psi^{(\alpha)}\right\|_{\LL^1} \right) = \frac{h(\mu) \alpha m_{b,1}\left\|\phi \right\|_{\LL^1}}{1-\alpha m_{b,1}\left\|\phi \right\|_{\LL^1}},
    \end{align*}
    where $ m_{b,1} = \int_{\RR} \left|b(x)\right| \vartheta(dx)$ and $\psi^{(\alpha)} = \sum\limits_{k\geq 1} \alpha^k \left|\phi\right|^{\ast k} $ and $\left|\phi\right|^{\ast (k+1)} = \left|\phi\right| \ast \left|\phi\right|^{\ast k}$, for $k\geq 1$.
\end{lemma}
\begin{proof}
    Let $s\in \RR_+$. Using the fact that $h$ is $\alpha$-Lipschitz, we get:
    \begin{align*}
        \E{\lambda_s} &\leq h(\mu) + \alpha \E{\int_{(0,s)\times \RR_+ \times \RR} \left| \phi(s-u)b(x)\right| \ind{\theta \leq \lambda_u} N(du,d\theta,dx)}\\
        &\leq h(\mu) + \alpha \int_0^s \int_\RR \left|\phi(s-u)\right| \E{\lambda_u} \left|b(x) \right| \vartheta(dx) du\\
        &\leq h(\mu) + \alpha \E{\left|b(X_1)\right|} \int_0^s \left|\phi(s-u)\right|\E{\lambda_u} du.
    \end{align*}
    Then, since $\alpha \E{\left|b(X_1)\right|} \left\| \phi \right\|_{\LL^1}<1$ (Assumption \ref{assump: noyau}), there exists $\psi^{(\alpha,b)} := \sum\limits_{k\geq 1} \alpha^k m_{b,1}^k \left|\phi\right|^{\ast k} $ with $\left|\phi\right|^{\ast (k+1)} = \left|\phi\right| \ast \left|\phi\right|^{\ast k}$ and $m_{b,1}=\E{|b(X_1)|}$ such that
    \begin{align*}
        \E{\lambda_s} \leq h(\mu)\left(1+ \left\| \psi^{(\alpha,b)}\right\|_{\LL^1} \right) = \frac{h(\mu) \alpha \E{\left|b(X_1)\right|} \left\|\phi \right\|_{\LL^1}}{1-\alpha \E{\left|b(X_1)\right|}\left\|\phi \right\|_{\LL^1}} = \frac{h(\mu) \alpha m_{b,1} \left\|\phi \right\|_{\LL^1}}{1-\alpha m_{b,1}\left\|\phi \right\|_{\LL^1}}. 
    \end{align*}
\end{proof}

\begin{lemma} \label{lem: maj esp cond}
    Under Assumption \ref{assump: noyau}, we have for any $(u,t,\rho,x) \in \RR_+^3\times \RR$
    $$\EE_u\left[\left| D_{(u,\rho,x)}\left(\lambda_t\right) \right| \right] \leq \1_{\rho \leq \lambda_u}\frac{\left|b(x) \right|}{m_{b,1}}\psi^{(\alpha,b)}(t-u) \quad a.s.$$
    where $m_{b,1} = \E{|b(X_1)|}$ and $\psi^{(\alpha,b)} := \sum_{k\geq 1} \alpha^k m_{b,1}^k \phi^{*k}$.
\end{lemma}
\begin{proof}
    First, for any $t\in \RR_+$, we denote by $\xi^b_t$ the following:
    $$\xi^b_t := \int_{(0,t)\times \RR_+ \times \RR} \phi(t-r)\ind{\theta \leq \lambda_r}b(y) N(dr,d\theta,dy).$$
    Then, by the definition of the Malliavin derivative, we have
    $$D_{(u,\rho,x)}\left(\xi_t^b \right) = \phi(t-u)\1_{\rho \leq \lambda_u}b(x) + \int_{(u,t)\times \RR_+\times \RR} \phi(t-r) \left(\1_{\theta \leq \lambda_r \circ \eps_{(u,\rho,x)}^{+}} - \1_{\theta\leq \lambda_r} \right) b(y)N(dr,d\theta,dy).$$
    Using the triangle inequality, we obtain
    $$\left| D_{(u,\rho,x)}\left(\xi_t^b \right) \right| \leq \left|\phi(t-u)\right|\1_{\rho \leq \lambda_u} \left| b(x)\right|+ \int_{(u,t)\times \RR_+ \times \RR} \left|\phi(t-r)\right| \left|\1_{\theta\leq \lambda_r} - \1_{\theta \leq \lambda_r\circ \eps_{(u,\rho)}^{+}} \right| \left| b(y)\right| N(dr,d\theta,dy).$$
    Then,
    $$\EE_u\left[\left| D_{(u,\rho,x)}\left(\xi_t^b \right) \right| \right] \leq \left|\phi(t-u)\right|\1_{\rho \leq \lambda_u}\left| b(x)\right| + m_{b,1}\int_u^t \left|\phi(t-r)\right| \EE_u\left[\left| D_{(u,\rho,x)}\left(\lambda_r \right) \right| \right] dr.$$
    Finally, we have that $y \mapsto \1_{u\leq y}\EE_u\left[\left|D_{(u,\rho,x)}\left(\lambda_y^\infty \right) \right| \right]$  is a subsolution of a Volterra equation since
    \begin{align*}
    \EE_u\left[\left| D_{(u,\rho,x)}\left(\lambda_t^\infty \right) \right| \right] &\leq \alpha \EE_u\left[\left| D_{(u,\rho,x)}\left(\xi_t^\infty \right) \right| \right]\\
    &\leq  \alpha \left|\phi(t-u)\right|\1_{\rho \leq \lambda_u^{\infty}} \left| b(x)\right| + \alpha m_{b,1} \int_u^t \left|\phi(t-r)\right| \EE_u\left[\left| D_{(u,\rho,x)}\left(\lambda_r^\infty \right) \right| \right] dr.
    \end{align*}
    According to Proposition A.1.4 \cite{coutin_normal_2024}, we define $\psi^{(\alpha,b)}$ has follows
    $$\psi^{(\alpha,b)} := \sum\limits_{k\geq 1} \alpha^k m_{b,1}^k\left|\phi \right|^{\ast k}$$
    where $\left| \phi\right|^{\ast (k+1)} = \left| \phi\right| \ast \left| \phi\right|^{\ast k}$ and get
\begin{align*}
    \EE_u\left[\left| D_{(u,\rho,x)}\left(\lambda_t \right) \right| \right] &\leq  \alpha \left|\phi(t-u)\right|\1_{\rho \leq \lambda_u} \left| b(x)\right|+ \alpha \left| b(x)\right| \int_u^t \psi^{(\alpha,b)}(t-r)\left|\phi(r-u)\right|\1_{\rho \leq \lambda_u} dr\\
    & \leq  \1_{\rho \leq \lambda_u} \frac{\left| b(x)\right|}{m_{b,1}}\left( \alpha m_{b,1}\left|\phi(t-u)\right| + \alpha m_{b,1}\int_u^t \psi^{(\alpha,b)}(t-r)\left|\phi(r-u)\right| dr \right) \\
    & \leq \1_{\rho \leq \lambda_u}\frac{\left| b(x)\right|}{m_{b,1}}\psi^{(\alpha,b)}(t-u).
\end{align*}
\end{proof}

\begin{lemma}
\label{lem: maj_sum_int_lambda_carré}
    Let $(t_i)_{i=0,\dots, n}$ be the subdivision defined in \ref{def: subdivision}. Then there exists $C>0$ such that for any $T>0$ and any $n\in \NN^\ast$,
    \begin{align*}
        \frac{1}{T^2} \sum_{i=0}^{n-1} \E{\left(\int_{t_i T}^{t_{i+1}T} \lambda_s ds\right)^2} \leq C\left( \frac{1}{T}+ \frac{1}{n}\right).
    \end{align*}
\end{lemma}
\begin{proof}
We begin this proof by noticing that for any $i=0,\dots,n-1$, we have
\begin{align*}
    \E{\left(\int_{t_i T}^{t_{i+1}T} \lambda_t dt\right)^2 } =& 2 \int_{t_i T}^{t_{i+1}T}\int_{t_i T}^s \E{\lambda_s \lambda_t} ds dt \nonumber \\
    =& \int_{t_i T}^{t_{i+1}T}\int_{t_i T}^s \E{\left(\lambda_s-\E{\lambda_s} \right)\left(\lambda_t-\E{\lambda_t} \right)} ds dt  + \int_{t_i T}^{t_{i+1}T}\int_{t_i T}^s \E{\lambda_s}\E{\lambda_t} ds dt .
\end{align*}
Lemma \ref{lem: maj_esp_lambda} yields 
\begin{align}
\label{eq: maj_int_prod_esp_lambda}
    \int_{t_i T}^{t_{i+1}T}\int_{t_i T}^s \E{\lambda_s}\E{\lambda_t} ds dt \leq \left(\frac{h(\mu) \alpha m_{b,1} \left\|\phi \right\|_{\LL^1}}{1-\alpha m_{b,1}\left\|\phi \right\|_{\LL^1}}\right)^2 \frac{T^2}{n^2}.
\end{align}

On the other hand, according to \cite{coutin_normal_2024} (see the upper bound of $T_{1,1}$ in the proof of Theorem 4.3.18 for more details), we have for $s<t$,
\begin{align*}
    \E{\left(\lambda_s-\E{\lambda_s} \right)\left(\lambda_t-\E{\lambda_t} \right)} &\leq \int_0^s \int_{\RR_+}\int_\RR  \EE_u\left[\left| D_{(u,\rho,x)}\left(\lambda_t \right) \right| \right]  \EE_u\left[\left| D_{(u,\rho,x)}\left(\lambda_s \right) \right| \right] du d\rho \vartheta(dx).
\end{align*}
And so, we get
\begin{align*}
    \E{\left(\lambda_s-\E{\lambda_s} \right)\left(\lambda_t-\E{\lambda_t} \right)} &\leq \frac{m_{b,2}}{\left|m_{b,1} \right|^2}\int_0^s \E{\lambda_u} \Psi^{(\alpha,b)}(t-u)\Psi^{(\alpha,b)}(s-u) du,
\end{align*}
where $m_{b,i}= \E{\left|b(X_1)\right|^i}$ and $\Psi^{(\alpha)}:x\mapsto \sum_{k\geq 1} \alpha^k \phi^{*k}(x)$. Thus, we get
\begin{align*}
    \int_{t_i T}^{t_{i+1}T}\int_{t_i T}^t \E{\left(\lambda_s-\E{\lambda_s} \right)\left(\lambda_t-\E{\lambda_t} \right)} ds dt &\leq \frac{m_{b,2}}{\left|m_{b,1} \right|^2} \int_{t_i T}^{t_{i+1}T}\int_{t_i T}^t \int_0^s \E{\lambda_u} \Psi^{(\alpha)}(t-u)\Psi^{(\alpha)}(s-u) du ds dt\\
    &\leq \frac{m_{b,2}}{\left|m_{b,1} \right|^2} \sup_{u\in [0,1]} \E{\lambda_u} \left\| \Psi^{(\alpha)}\right\|_{\LL^1}^2 T(t_i - t_{i+1}).
\end{align*}
Therefore, by Lemma \ref{lem: maj_esp_lambda} and using the fact that $\left\| \Psi^{(\alpha)}\right\|_{\LL^1} \leq \frac{\alpha m_{b,1}\left\|\phi \right\|_{\LL^1}}{1-\alpha m_{b,1} \left\|\phi \right\|_{\LL^1}}$, we obtain
\begin{align}
\label{eq: maj_int_cov_lambda}
    \frac{1}{T^2} \sum_{i=0}^{n-1}\int_{t_i T}^{t_{i+1}T}\int_{t_i T}^t \E{\left(\lambda_s-\E{\lambda_s} \right)\left(\lambda_t-\E{\lambda_t} \right)} ds dt \leq h(\mu)m_{b,2}m_{b,1}\left(\frac{\alpha \left\|\phi \right\|_{\LL^1}}{1-\alpha m_{b,1} \left\|\phi \right\|_{\LL^1}}\right)^3 \frac{1}{T} 
\end{align}
Combining \eqref{eq: maj_int_cov_lambda} and \eqref{eq: maj_int_prod_esp_lambda}, we get the result.
\end{proof}
Finally, we end this section by a last technical lemma. Note that the proof of this lemma will strongly relies on the proof of Theorem 4.3.18 in \cite{coutin_normal_2024}. In particular, we adapt some inequality to our case.
\begin{lemma}
\label{lem: maj_sum_diff_variance}
    For any $n\in \NN^\ast$, 
    \begin{align*}
        \sum_{i=1}^n \E{\left|\frac{\sigma^2}{n} - \frac{1}{T}\int_{t_{i-1} T}^{t_{i}T} \lambda_s ds \right|} \leq \frac{C}{T} + C\sqrt{\frac{n}{T}}.
    \end{align*}
\end{lemma}
\begin{proof}
    Let $n\in \NN^\ast$. Since $\E{ \lambda^{\infty}_{\cdot} } \equiv \sigma^2$, it holds that
    \begin{align*}
        \sum_{i=1}^n \E{\left|\sigma^2(t_i-t_{i-1})-\frac{1}{T}\int_{t_{i-1}T}^{t_i T} \lambda_s ds\right|}\leq& \left( \sum_{i=1}^n \frac{1}{T}\E{\left|\int_{t_{i-1}T}^{t_i T} \E{\lambda^{\infty}_s}-\lambda^{\infty}_s ds \right|} \right) + \left(\frac{1}{T}\int_{0}^{T} \E{\left| \lambda^{\infty}_s-\lambda_t\right|} ds \right)
    \end{align*}
    Now, an adaptation of the Proof of the upper bound of $T_{1,2}$ in \cite{coutin_normal_2024} gives
    \begin{align*}
        \int_{0}^{T} \E{\left| \lambda^{\infty}_s-\lambda_t\right|} ds  \leq \frac{\alpha m_{b,1}}{1-\alpha m_{b,1}\left\| \phi\right\|_{\LL^1}} \sup_{t\in \RR_+} \E{\lambda_s} \int_0^{\infty} t\phi(t) dt.
    \end{align*}
    Thus, by Lemma \ref{lem: maj_esp_lambda}, we get 
    \begin{align}
    \label{eq: maj_int_L1_diff_lambda}
       \int_{0}^{T} \E{\left| \lambda^{\infty}_s-\lambda_t\right|} ds  \leq \left(\frac{\alpha m_{b,1}}{1-\alpha m_{b,1}\left\| \phi\right\|_{\LL^1}} \right)^2 h(\mu)\left\| \phi\right\|_{\LL^1} \int_0^{\infty} t\phi(t) dt.
    \end{align}
    Note that the upper bound is a constant thanks to Assumption \ref{assump: noyau}.\\
    Besides, we also have
    \begin{align*}
        \frac{1}{T}\E{\left|\int_{t_{i-1}T}^{t_i T} \E{\lambda^{\infty}_s}-\lambda^{\infty}_s ds \right|} &\leq  \frac{1}{T} \sqrt{\E{\left|\int_{t_{i-1}T}^{t_i T} \E{\lambda^{\infty}_s}-\lambda^{\infty}_s ds \right|^2}}\\
        &\leq \frac{1}{T} \sqrt{\E{2\int_{t_{i-1}T}^{t_i T} \int_{t_{i-1}T}^{s} (\E{\lambda^{\infty}_s}-\lambda^{\infty}_s)(\E{\lambda^{\infty}_t}-\lambda^{\infty}_t)ds dt}}\\
        &\leq \sqrt{2 \sigma^2 \frac{m_{b,2}\left\| \psi^{(\alpha,b)}\right\|_{\LL^1(0,1)}^2}{|m_{b,1}|^2 Tn} }.
    \end{align*}
    where $m_{b,i}=\E{|b(X_1)|^i}$ and $\psi^{(\alpha,b)} = \sum\limits_{k\geq 1} \alpha^k m_{b,1}^k\left|\phi\right|^{\ast k} $ and $\left|\phi\right|^{\ast (k+1)} = \left|\phi\right| \ast \left|\phi\right|^{\ast k}$, for $k\geq 1$. And so, we obtain
    \begin{align}
    \label{eq: eq_maj_int_discret_station_lambda_esp}
        \frac{1}{T}\E{\left|\int_{t_{i-1}T}^{t_i T} \E{\lambda^{\infty}_s}-\lambda^{\infty}_s ds \right|} &\leq  \frac{\alpha \sigma \sqrt{2 m_{b,2}} \left\| \phi\right\|_{\LL^1}}{1-\alpha m_{b,1}\left\| \phi\right\|_{\LL^1}}  \frac{1}{\sqrt{Tn}}.
    \end{align}
    Finally, combining \eqref{eq: maj_int_L1_diff_lambda} and \eqref{eq: eq_maj_int_discret_station_lambda_esp}, we have that there exists $C>0$ independent of $T$ and $n\in \NN^\ast$ such that
    \begin{align*}
        \sum_{i=1}^n \E{\left|\frac{\sigma^2}{n} - \frac{1}{T}\int_{t_{i-1} T}^{t_{i}T} \lambda_s ds \right|} \leq \frac{C}{T} + C\sqrt{\frac{n}{T}}.
    \end{align*}
\end{proof}

\subsection{The multidimensional Malliavin-Stein method}

\label{subsec: Ornstein_semigroup}
We start this section by recalling a standard result regarding the Stein equation. 
\begin{proposition}
\label{prop: Stein_eq}
    For $\bG$ a centered gaussian vector with covariance matrix $\bsig$ and $f\in \Ccal^2(\RR^n)$, the Ornstein–Uhlenbeck semigroup
\begin{align*}
    P_t f: \bx \in \RR^n \mapsto \E{f\left( e^{-t}\bx + \sqrt{1-e^{-2t}} \bG \right)}.
\end{align*}
is solution to the Stein equation, i.e.
\begin{align}
\label{eq: Stein}
    f(\bx) - \E{f(\bG)} = \int_0^\infty \left\langle \bx, \nabla P_t f(\bx)\right\rangle_{\RR^n}-\left\langle \bsig, \nabla^2 P_t f(\bx) \right\rangle_{H.S.} dt, \quad \bx \in \RR^n.
\end{align}
\end{proposition}
The proof and some generalizations of \eqref{prop: Stein_eq} can be found in \cite{nualart_malliavin_1995} (Propositions 1.4.4 and 1.4.5). 
Inspired by the work of \cite{nourdin_multivariate_2010}, \cite{huang_rate_2024} or \cite{coutin_donskers_2020}, among others, we focus on the integral and establish an upper bound by using Malliavin-Stein's method. The following proposition is therefore an adaptation of the different existing results in the context of this paper. Let us remind to the reader some notations that will be in order for the next results:
\begin{align*}
    F^{(T)}_t = F_{tT}= \frac{H_{tT}-\int_0^{tT}\lambda_s ds}{\sqrt{T}},\quad 
    \bdF^T = \left(F_{t_1T}-F_{t_0T}, \dots, F_{t_n T}-F_{t_{n-1}T} \right),\quad
    \bG \sim \Ncal \left(0_{\RR^n}, \bsig \right),
\end{align*}
with $\bsig = \frac{\sigma^2}{n}\left(\int_\RR \left|g(x)\right|^2 \vartheta(dx) \right) \mathbf{I_n}= m_{g,2}\frac{\sigma^2}{n}\mathbf{I_n}$.
\begin{proposition}[Malliavin-Stein's method]
\label{prop: malliavin_stein_calcul}
    For $v\in \Ccal^3(\RR^n)$ with bounded derivatives and $T>0$
    \begin{align*}
        &\left|\E{\left\langle \bdF^{(T)}, \nabla v\left( \bdF^{(T)}\right)\right\rangle_{\RR^n}-\left\langle \bsig, \nabla^2 v\left( \bdF^{(T)}\right) \right\rangle_{H.S.}}\right|\\
        &\leq  \frac{1}{2}\max_{i=1,\dots,n}\left\|\partial_{iii}^3 v \right\|_{\infty} \frac{m_{g,3}}{T^{3/2}} \int_0^T \E{\lambda_s}ds + m_{g,2}\max_{i=1,\dots,n}\left\|\partial_{ii}^2 v \right\|_{\infty} \sum_{i=1}^n \E{\left|\frac{\sigma^2}{n}-\frac{1}{T}\int_{t_{i-1}T}^{t_i T} \lambda_s ds\right|},
    \end{align*}
    with $m_{g,i}= \int_\RR \left| g(x)\right|^i \vartheta(dx)$, $i=2,3$.
\end{proposition}

\begin{proof}
For the sake of simplicity, we will use the following notation:
\begin{align*}
    F_i := F_{t_i T}, \quad i=1, \dots, n.
\end{align*}
We first notice that
\begin{align*}
    \E{\left\langle \bdF^{(T)}, \nabla v\left( \bdF^{(T)}\right)\right\rangle_{\RR^n}} =& \sum_{i=1}^n \E{(F_i-F_{i-1})\partial_i v\left( \bdF^{(T)}\right)} =\sum_{i=1}^n \E{\delta\left(\Zcal Z^{(i)} \right)\partial_i v\left( \bdF^{(T)}\right)},
\end{align*}
where
\begin{align*}
    \Zcal_{(t,\theta)} := \1_{\theta \leq \lambda_t} \quad \text{and} \quad Z^{(i)}_{(t,x)} := \frac{\1_{(t_{i-1}T,t_i T]}(t)}{\sqrt{T}}g(x), \quad i=1,\dots ,n.
\end{align*}
So, Integrating by part formula of Proposition \ref{prop:IPP} yields
\begin{align*}
    \E{\left\langle \bdF^{(T)}, \nabla v\left( \bdF^{(T)}\right)\right\rangle_{\RR^n}} =& \sum_{i=1}^n \int_{\RR_+^2\times \RR}\E{\Zcal_{(t,\theta)} Z^{(i)}_{(t,x)} D_{(t,\theta,x)}\left(\partial_i v\left( \bdF^{(T)}\right) \right)} dtd\theta \vartheta(dx).
\end{align*}
Since $D_{(t,\theta,x)}\left(\partial_i v\left( \bdF^{(T)}\right) \right) = \partial_i v\left( \bdF^{(T)} \circ \eps_{(t,\theta,x)}^+\right) - \partial_i v\left( \bdF^{(T)}\right)$, we can use Taylor's expansion formula to get:
\begin{align*}
    D_{(t,\theta,x)}\left(\partial_i v\left( \bdF^{(T)}\right) \right) =& \sum_{j=1}^n D_{(t,\theta,x)}F_j \partial^2_{ji}v\left( \bdF^{(T)}\right) \\
    &+ \int_0^1 (1-s) \sum_{j=1}^n \sum_{k=1}^n \partial^3_{ijk}v\left(\bdF^{(T)}+sD_{(t,\theta,x)}\left( \bdF^{(T)}\right)\right) D_{(t,\theta,x)}(F_j) D_{(t,\theta,x)}(F_k) ds.
\end{align*}
Thus, 
\begin{align}
\label{eq: eq_prod_scal}
    \E{\left\langle \bdF^{(T)}, \nabla v\left( \bdF^{(T)}\right)\right\rangle_{\RR^n}} &= \sum_{i=1}^n \int_{\RR_+^2\times \RR}\E{\Zcal_{(t,\theta)} Z^{(i)}_{(t,x)} \sum_{j=1}^n D_{(t,\theta,x)}(F_j) \partial^2_{ji}v\left( \bdF^{(T)}\right)} dtd\theta \vartheta(dx) \nonumber \\
    &+ \sum_{i=1}^n \int_{\RR_+^2\times \RR}\E{\Zcal_{(t,\theta)} Z^{(i)}_{(t,x)} \sum_{j=1}^n \sum_{k=1}^n  D_{(t,\theta,x)}(F_j) D_{(t,\theta,x)}(F_k) \hat{F}_{ijk}^T(t)} dt d\theta \vartheta(dx),
\end{align}
where we have used the notation
\begin{align*}
    \hat{F}_{ijk}^T(t) := \int_0^1 (1-s) \partial^3_{ijk} v\left(\bdF^{(T)}+sD_{(t,\theta,x)}\left(\bdF^{(T)}\right)\right) ds
\end{align*}
We first focus on the last term that we denote by $A_2^{T}$. Using Proposition \ref{prop: heisenberg}, stated that
\begin{align}\label{eq:Heisenberg}
        D_{(t,\theta,x)}\left( F_j\right) &= D_{(t,\theta,x)}\left( \delta \left( \Zcal Z^{(j)} \right) \right) = \Zcal_{(t,\theta)} Z_{(t,x)}^{(j)} + \delta \left( Z^{(j)} D_{(t,\theta,x)}\left( \Zcal\right) \right),
\end{align}
we compute
\begin{align*}
    A_2^T =& \sum_{i,j,k=1}^n \int_{\RR_+^2\times \RR}\E{\Zcal_{(t,\theta)} Z^{(i)}_{(t,x)}  D_{(t,\theta,x)}(F_j) D_{(t,\theta,x)}(F_k) \hat{F}_{ijk}^T(t)} dt d\theta \vartheta(dx)\\
    =&\sum_{i=1}^n \int_{\RR_+^2 \times \RR}\E{\Zcal_{(t,\theta)} \left\{Z^{(i)}_{(t,x)} \right\}^3  \hat{F}_{iii}^T(t)} dt d\theta \vartheta(dx)\\
    &+ \sum_{i,j=1}^n \int_{\RR_+^2\times \RR}\E{\Zcal_{(t,\theta)} \left|Z^{(i)}_{(t,x)}\right|^2  \delta \left( Z^{(j)} D_{(t,\theta,x)}\left( \Zcal\right) \right)\hat{F}_{iji}^T(t)} dt d\theta \vartheta(dx)\\
    &+ \sum_{i,k=1}^n \int_{\RR_+^2 \times \RR}\E{\Zcal_{(t,\theta)} \left|Z^{(i)}_{(t,x)}\right|^2  \delta \left( Z^{(k)} D_{(t,\theta,x)}\left( \Zcal\right) \right)\hat{F}_{iki}^T(t)} dt d\theta \vartheta(dx) \\
    &+ \sum_{i,j,k=1}^n \int_{\RR_+^2 \times \RR}\E{\Zcal_{(t,\theta)} Z^{(i)}_{(t,x)}   \delta \left( Z^{(j)} D_{(t,\theta,x)}\left( \Zcal\right) \right)\delta \left( Z^{(k)} D_{(t,\theta,x)}\left( \Zcal\right) \right)\hat{F}_{ijk}^T(t)} dt d\theta \vartheta(dx) .
\end{align*}
So we have:
\begin{align*}
    A_2^T =&\sum_{i=1}^n \int_{\RR_+^2\times \RR}\E{\Zcal_{(t,\theta)} \left\{Z^{(i)}_{(t,x)} \right\}^3  \hat{F}_{iii}^T(t)} dt d\theta \vartheta(dx) \\
    &+ 2\sum_{i,j=1}^n \int_{\RR_+^2 \times \RR}\E{\Zcal_{(t,\theta)} \left|Z^{(i)}_{(t,x)}\right|^2  \delta \left( Z^{(j)} D_{(t,\theta,x)}\left( \Zcal\right) \right)\hat{F}_{iji}^T(t)} dt d\theta \vartheta(dx)\\
    &+ \sum_{i,j,k=1}^n \int_{\RR_+^2\times \RR}\E{\Zcal_{(t,\theta)} Z^{(i)}_{(t,x)}   \delta \left( Z^{(j)} D_{(t,\theta,x)}\left( \Zcal\right) \right)\delta \left( Z^{(k)} D_{(t,\theta,x)}\left( \Zcal\right) \right)\hat{F}_{ijk}^T(t)} dt d\theta \vartheta(dx)
\end{align*}
Besides, as for Lemma 8.2 of \cite{coutin_normal_2024}, the two last terms are null.
So, we obtain
\begin{align*}
    A_2^T =&\sum_{i=1}^n \int_{\RR_+^2\times \RR}\E{\Zcal_{(t,\theta)} \left\{Z^{(i)}_{(t,x)} \right\}^3  \hat{F}_{iii}^T(t)} dt d\theta \vartheta(dx).
\end{align*}
Since $\sup_{t\in [0,1]}\left| \hat{F}_{iii}^T(t)\right| \leq \frac{1}{2} \max_{i=1,\dots,n}\left\| \partial^3_{iii} v\right\|_{\LL^\infty(\RR^n)}$, we get:
\begin{align}
    A_2^T &\leq \frac{1}{2} \max_{i=1,\dots,n}\left\| \partial^3_{iii} v\right\|_{\LL^\infty(\RR^n)} \times \frac{1}{T^{3/2}} \left(\int_\RR \left|g(x)\right|^3 \vartheta(dx)\right) \sum_{i=1}^n  \int_{t_{i-1}T}^{t_i T} \E{\lambda_s} ds \nonumber \\
    &\leq \frac{1}{2} \max_{i=1,\dots,n}\left\| \partial^3_{iii} v\right\|_{\LL^\infty(0,1)} \times\frac{m_{g,3}}{T^{3/2}} \int_0^T \E{\lambda_s} ds.
    \label{eq: maj_A_2}
\end{align}
We now focus on the first term of equality, that is
\begin{align*}
    A_1^T := \sum_{i=1}^n \int_{\RR_+^2\times \RR}\E{\Zcal_{(t,\theta)} Z^{(i)}_{(t,x)} \sum_{j=1}^n D_{(t,\theta,x)}(F_j) \partial^2_{ji}v\left( \bdF^{(T)}\right)} dtd\theta \vartheta(dx).
\end{align*}
By the Heisenberg inequality \eqref{eq:Heisenberg}, we have
\begin{align*}
    A^T_1 =& \sum_{i=1}^n  \sum_{j=1}^n \int_{\RR_+^2\times \RR}\E{\Zcal_{(t,\theta)} Z^{(i)}_{(t,x)} Z_{(t,x)}^{(j)} \partial^2_{ji}v\left(\bdF^{(T)}\right)} dtd\theta \vartheta(dx) \\
    &+ \sum_{i=1}^n  \sum_{j=1}^n \int_{\RR_+^2\times \RR}\E{\Zcal_{(t,\theta)} Z^{(i)}_{(t,x)} \delta \left( Z^{(j)} D_{(t,\theta,x)}\left( \Zcal\right)  \right) \partial^2_{ji}v\left(\bdF^{(T)}\right)} dtd\theta \vartheta(dx).
\end{align*}
Then, an adaptation of Lemma 8.2 of \cite{coutin_normal_2024} yields
\begin{align*}
    A^T_1 =& \sum_{i=1}^n  \sum_{j=1}^n \int_{\RR_+^2\times \RR}\E{\Zcal_{(t,\theta)} Z^{(i)}_{(t,x)} Z_{(t,x)}^{(j)} \partial^2_{ji}v\left(\bdF^{(T)}\right)} dtd\theta \vartheta(dx)= \frac{m_{g,2}}{T}\E{\sum_{i=1}^n \partial^2_{ii}v\left(\bdF^{(T)}\right)  \int_{t_{i-1}T}^{t_i T} \lambda_s ds  }.
\end{align*}
Hence,
\begin{align}
    \left|A_1^T - \E{\left\langle \bsig, \nabla^2 v\left( \bdF^{(T)}\right) \right\rangle_{H.S.}}\right| &=\left| \sum_{i=1}^n \E{\partial_{ii}^2v\left(\bdF^{(T)}\right) \left\{\sigma^2 m_{g,2} \left(t_i-t_{i-1} \right)-\frac{m_{g,2}}{T}\int_{t_{i-1}T}^{t_i T} \lambda_s ds\right\}}\right| \nonumber\\
    &\leq m_{g,2}\max_{i=1,\dots,n} \left\|\partial^2_{ii}v \right\|_{\LL^\infty(\RR^n)}\left| \sum_{i=1}^n \E{\left\{\sigma^2 \left(t_i-t_{i-1} \right)-\frac{1}{T}\int_{t_{i-1}T}^{t_i T} \lambda_s ds\right\}}\right| .
    \label{eq: maj_A1_prod_scal_HS}
\end{align}
\end{proof}

\subsection{Results on the Ornstein-Uhlenbeck semigroup}
In order to apply Proposition \ref{prop: malliavin_stein_calcul} without restriction, we give in the following two lemmas: the first one guarantees the smoothing effect of the semi group and the second one gives explicit upper bound for its partial derivatives. Remark that these results are inspired by the work of \cite{coutin_donskers_2020} where similar result can be found.

\begin{lemma}
    For every $f\in \Ccal_1$, $P_t f \in \Ccal^{\infty}$ and we have:
    \begin{align*}
        \partial_i P_t f(\bx) =& \frac{e^{-t}}{\sigma_n^{2}\beta_t} \E{G_i f\left(e^{-t}\bx+\beta_t \bG \right)},\quad \forall i=1,\dots,n,\\
        \partial^2_{ij} P_t f(\bx) =& \frac{e^{-3t/2}}{ \sigma_n^{4}\beta_{t/2}^2} \E{G_j \widetilde{G}_i f\left( e^{-t} \bx + \beta_{t/2}e^{-t/2}\widetilde{\bG} + \beta_{t/2} \bG\right)},\quad \forall i,j=1,\dots,n,
    \end{align*}
    where $\sigma_n = \frac{\Tilde{\sigma}}{\sqrt{n}}=\frac{\sigma}{\sqrt{n}}\sqrt{m_{g,2}}$, $\beta_t = \sqrt{1-e^{-2t}}$ and $\widetilde{\bG}$ is an independent copy of $\bG$.
\end{lemma}

\begin{proof}
    In this proof, we only prove that $P_tf$ is a $\Ccal^2$ function. However, the methodology to prove that $P_tf$ is a $\Ccal^\infty$ function is an iteration of our method.\\
    Let $f\in \Ccal_2$. Then, for any $\bx\in \RR^n$,
    \begin{align*}
        \partial_i P_t f (\bx) = e^{-t} P_t\left( \partial_i f\right)(\bx) =e^{-t} \E{\partial_i f\left( e^{-t}\bx+\beta_t \bG\right)} = \frac{e^{-t}}{\sigma_n^2 \beta_t} \E{G_i f\left(e^{-t}\bx+\beta_t \bG \right)}.
    \end{align*}
    Here, we have used the Stein Lemma for the last equality.\\
    Since $P_t f = P_{t/2} \left( P_{t/2} f \right)$, we have:
    \begin{align*}
    \partial_j \partial_i P_t f(\bx) = e^{-t/2} \partial_j P_{t/2} \left\{ \partial_i P_{t/2}f \right\}(\bx) = \frac{e^{-t}}{\sigma_n^2 \beta_{t/2}} \E{G_j \partial_i P_{t/2}f\left( e^{-t/2} \bx + \beta_{t/2} \bG\right) }.
    \end{align*}
    We now make use of the first result to get:
    \begin{align*}
    \partial_j \partial_i P_t f(\bx) &= \frac{e^{-3t/2}}{\sigma_n^4 \beta_{t/2}^2} \E{G_j \widetilde{G}_i f\left( e^{-t/2} \left( e^{-t/2}\bx + \beta_{t/2}\widetilde{\bG}\right) + \beta_{t/2} \bG\right) }\\
    &= \frac{e^{-3t/2}}{\sigma_n^4 \beta_{t/2}^2} \E{G_j \widetilde{G}_i f\left( e^{-t} \bx + \beta_{t/2}e^{-t/2}\widetilde{\bG} + \beta_{t/2} \bG\right)}.
\end{align*}
By similar computations, we can prove that $P_t f\in \Ccal^{\infty}$.
\end{proof}

\begin{lemma} 
\label{lem: maj_derive_semigroup}
Let $f\in \Ccal^1(\RR^n)$ such that $\max_{i=1,\dots, n} \left\| \partial_i f\right\|_{\LL^\infty(\RR^n)} \leq 1$.
For any $t\in \RR_+$ and $n\in \NN^\ast$,
    \begin{align*}
        \max_{i=1,\dots,n}\left\|\partial_i P_t f \right\|_{\LL^\infty(\RR^n)} \leq e^{-t},\quad & \max_{i,j=1,\dots,n}\left\|\partial_{ij}^2 P_t f \right\|_{\LL^\infty(\RR^n)} \leq \frac{e^{-3t/2}}{\sigma_n\beta_{t/2}}, \quad&\max_{i,j,k=1,\dots,n} \left\| \partial_{ijk}^3 P_t f \right\|_{\LL^\infty(\RR^n)} \leq \frac{e^{-5t/2}}{\sigma_n^2\beta_{t/2}^2}
    \end{align*}
\end{lemma}
\begin{proof}
    Fix $i,j,k=1,\dots,n$ and $\bx\in \RR^n$. We have:
    \begin{align*}
        \left|\partial_i P_t f(\bx)\right| \leq e^{-t} \E{\left| \partial_i f \left( e^{-t}\bx + \beta_t \bG \right)\right|} \leq e^{-t} \left\| \partial_i f \right\|_{\LL^\infty(\RR^n)} \leq e^{-t}
    \end{align*}
    On the other hand, we have
    \begin{align*}
   \left|\partial_j \partial_i P_t f (\bx)\right| &= \frac{e^{-t}}{\sigma_n^2 \beta_{t/2}} \left|\E{G_j \partial_i P_{t/2}f \left( e^{-t/2} \bx + \beta_{t/2} \bG\right) } \right| \leq \frac{e^{-t}}{\sigma_n^2 \beta_{t/2}} \left\|\partial_i P_{t/2} f \right\|_{\LL^\infty(\RR^n)}\E{\left|G_j\right|}.
    \end{align*}
    Since $\E{\left|G_j\right|}\leq \sigma_n$, we obtain
    \begin{align*}
        \left|\partial_j \partial_i P_t f(\bx)\right| &\leq \frac{e^{-3t/2}}{\sigma_n\beta_{t/2}}.
    \end{align*}
    Similarly, we have
    \begin{align*}
        \partial_{k}\partial_j \partial_i P_t f(\bx) = \frac{e^{-5t/2}}{\sigma_n^4 \beta_{t/2}^2}\E{G_j \widetilde{G}_i \partial_k f \left( e^{-t} \bx + \beta_{t/2}e^{-t/2}\widetilde{\bG} + \beta_{t/2} \bG\right)},
    \end{align*}
    and so,
    \begin{align*}
        \left|\partial_{kji}^3 P_t f(\bx)\right| \leq \frac{e^{-5t/2}}{\sigma_n^4 \beta_{t/2}^2} \left\| \partial_k f \right\|_{\LL^\infty(\RR^n)} \E{\left|G_j \widetilde{G}_i \right|} \leq \frac{e^{-5t/2}}{\sigma_n^2\beta_{t/2}^2}.
    \end{align*}
\end{proof}

\paragraph{Acknowledgment}: This work received support from the University Research School EUR-MINT
(State support managed by the National Research Agency for Future Investments
 program bearing the reference ANR-18-EURE-0023).

\begin{sloppypar}
\printbibliography
\end{sloppypar}

\end{document}